\documentclass[10pt]{article}

\usepackage{latexsym}
\usepackage{amssymb}
\usepackage{amsthm}
\usepackage{amscd}
\usepackage{amsmath}
\usepackage{tikz}
\usetikzlibrary{patterns}
\usepackage{mathabx}
\usepackage{stmaryrd}
\usepackage{shuffle}
\usepackage{comment}
\usepackage{graphicx}

\usepgflibrary{arrows}

\newtheorem{theorem}{Theorem}[section]

\newtheorem{lemma}[theorem]{Lemma}

\newtheorem{proposition}[theorem]{Proposition}
\newtheorem{corollary}[theorem]{Corollary}

\theoremstyle{definition}
\newtheorem*{example}{Example}

\newtheorem{remark}[theorem]{Remark}

\numberwithin{equation}{section}

\newcommand{\CC}{\mathbb{C}} 
\newcommand{\FF}{\mathbb{F}}

\newcommand{\ZZ}{\mathbb{Z}}

\newcommand{\cB}{\mathcal{B}}

\newcommand{\cK}{\mathcal{K}}
\newcommand{\cL}{\mathcal{L}}

\newcommand{\cR}{\mathcal{R}}

\newcommand{\fkut}{\mathfrak{ut}}
\newcommand{\fkl}{\mathfrak{l}}

\newcommand{\fkp}{\mathfrak{p}}
\newcommand{\fkr}{\mathfrak{r}}

\newcommand{\Ind}{\mathrm{Ind}}
\newcommand{\GL}{\mathrm{GL}}

\newcommand{\Inf}{\mathrm{Inf}} 

\newcommand{\Irr}{\mathrm{Irr}}

\newcommand{\UT}{\mathrm{UT}}

\newcommand{\Def}{\mathrm{Def}}

\newcommand{\Sym}{\mathrm{Sym}}

\newcommand{\Coeff}{\mathrm{Coeff}}

\newcommand{\scf}{\mathrm{scf}}

\newcommand{\cf}{\mathrm{cf}}

\newcommand{\dd}{\displaystyle}
\newcommand{\scs}{\scriptstyle}
\newcommand{\scscs}{\scriptscriptstyle}

\newcommand{\spanning}{\textnormal{-span}}
\newcommand{\One}{{1\hspace{-.75ex} 1}}

\def\adots{\mathinner{\mkern2mu\raise0pt\hbox{.}  % antidiagonal dots
\mkern2mu\raise4pt\hbox{.}\mkern1mu
\raise7pt\vbox{\kern7pt\hbox{.}}\mkern1mu}}

\newcommand{\Col}{\mathrm{Col}}
\newcommand{\Dela}{\mathrm{Dela}}

\newcommand{\Cl}{\mathtt{Cl}}
\newcommand{\Ch}{\mathtt{Ch}}
\newcommand{\reg}{\mathrm{reg}}
\newcommand{\FQSym}{\mathrm{FQSym}}

\newcommand{\deshuffle}{\mathrm{dSh}}
\newcommand{\remove}{\mathrm{rm}}

\newcommand{\coverings}{\mathrm{CInv}}
\newcommand{\coveringsets}{\mathrm{CInvS}}
\newcommand{\coreset}{\mathrm{crSt}}
\newcommand{\core}{\mathrm{core}}

\newcommand{\NCSym}{\mathrm{NCSym}}
\newcommand{\NSym}{\mathrm{NSym}}
\newcommand{\possibilities}{\mathrm{ccs}}
\newcommand{\Sfl}{\mathrm{Stfl}}
\newcommand{\faith}{\mathrm{fth}}
\newcommand{\Str}{\mathrm{Str}}
\newcommand{\Ac}{A'}

\makeatletter
\renewcommand{\@makefnmark}{\mbox{\textsuperscript{}}}
\makeatother

\allowdisplaybreaks[1]

\title{A categorification of the Malvenuto--Reutenauer\\ algebra via a tower of groups}

\date{}
\author{Farid Aliniaeifard and Nathaniel Thiem}

\begin{document}

\maketitle

\begin{abstract}
There is a long tradition of categorifying combinatorial Hopf algebras by the modules of a tower of algebras (or even better via the representation theory of a tower of groups).  From the point of view of combinatorics, such a categorification supplies canonical bases, inner products, and a natural avenue to prove positivity results.  Recent ideas in supercharacter theory have made fashioning the representation theory of a tower of groups into a Hopf structure more tractable.    This paper applies such a program to the Malvenuto--Reutenauer Hopf algebra.  In particular, we design functors on the representation theory of a tower of $p$-groups that realize the Hopf structure of the Malvenuto--Reutenauer algebra in such a way that its well-known fundamental basis corresponds to a supercharacter basis.
\end{abstract}

\section{Introduction\protect\footnote{Keywords: Hopf structure, permutations, supercharacters, categorification}}

The categorification of the Hopf algebra of symmetric functions $\Sym$ by the representation theory of the symmetric group is a foundational result in combinatorial representation theory.  There are other classical cases with similar constructions---as outlined in Macdonald \cite{Mac}---coming from wreath products and the finite general linear groups.  However, all these examples give Hopf algebras that are essentially copies of $\Sym$ (or a PSH algebra in the language of Zelevinsky \cite{Ze}).  Results like \cite{BLL} indicate that categorifying other Hopf algebras (with the usual induction and restriction functors)  may  require dispensing with towers of groups in favor of towers of algebras.  

The paper \cite{AIM} took a different approach to towers of groups by replacing the full representation theory with a supercharacter theory  and the traditional induction/restriction functor pairing with new functor combinations.  In this way, \cite{AIM} was able to categorify the symmetric functions in noncommuting variables $\NCSym$, and a similar approach in \cite{AT2} found a categorification of a Catalan Hopf subalgebra.   The underlying algebraic structure turns out to be  the shadow of a Hopf monoid \cite{ABT}, a generalization that better captures the underlying representation theory.    

While \cite{AIM} and \cite{AT2} began with the representation theory of finite unipotent uppertriangular groups $\UT_n(\FF_q)$ and found a Hopf structure, this paper was motivated by the opposite approach.  That is, we wanted to find a tower of groups with an associated supercharacter theory that would give us a non-commutative and non-cocommutative Hopf algebra.  As a test case, we selected the Malvenuto--Reutenauer Hopf algebra 
$$\FQSym=\bigoplus_{n\geq 0} \FQSym_n,$$
a graded self-dual Hopf algebra where each graded degree satisfies $\dim(\FQSym_n)=n!$. 

The Malvenuto--Reutenauer Hopf algebra was introduced by Malvenuto in \cite{Malv} with a basis called the fundamental basis.  It contains many well-known Hopf algebras, such as the Hopf algebra of symmetric functions $\Sym$ \cite{Mac}, the Hopf algebra of non-commutative symmetric functions $\NSym$ \cite{GKal}, Stembridge's peak algebra $\mathfrak{P}$ \cite{Stem97}, and the Loday--Ronco Hopf algebra of planar trees $\mathrm{LR}$ \cite{LR98}. Moreover, the Hopf algebra of quasi-symmetric functions  $\mathrm{QSym}$ is a quotient of $\FQSym$ \cite{Ges}. Aguiar and Sottile \cite{AS05} studied the structure of Malvenuto--Reutenauer Hopf algebra and produced a new basis, called the monomial basis, related to the fundamental basis by M\"{o}bius inversion on the weak order on the symmetric groups. They give a geometric description of the monomial basis product structure constants.   This paper studies a similar basis using a different order on permutations that arises naturally in our setting (see Section \ref{PermutationCharacterStructure}).    

A supercharacter theory is a framework developed by Diaconis--Isaacs \cite{DI} to study the representation theory of a group without requiring full knowledge of the irreducible characters.  In general, while groups have many such theories,  there are not many known constructions that work for arbitrary groups.  This paper uses the normal lattice supercharacter theory of a finite group $G$ developed by Aliniaeifard in \cite{Al}. This theory assigns a supercharacter theory to every sublattice of normal subgroups of $G$.  Our first goal was to find a tower of groups with associated sublattices of normal subgroups indexed by permutations.   We further decided to focus on abelian groups, since these tend to have more normal subgroups, and we settled on the Lie algebra $\fkut_n(\FF_q)$ of  $\UT_n(\FF_q)$ viewed as a finite additive group.  As an elementary abelian group $\fkut_n(\FF_q)$ has a fairly uninspiring group structure, but  the normal lattice supercharacter theory makes the group far more combinatorially compelling.  

The main result (Corollary \ref{MainIsomorphism}) of Section \ref{HopfAlgebraIsomorphism} finds a Hopf algebra isomorphism between $\FQSym$ and a representation theoretic algebra
$$\scf(\fkut)=\bigoplus_{n\geq 0}\scf(\fkut_n).$$
A key component of this isomorphism is to identify the functors $\Sfl$ and $\Dela$ that encode the Hopf structure of $\FQSym$ in $\scf(\fkut)$. Here, we take advantage of the feature that every supercharacter theory identifies two canonical bases: the superclass identifier basis and the supercharacter basis.  By computing the structure constants for the supercharacter basis in Theorems \ref{GoingUpCombinatorics} and \ref{GoingDownCombinatorics}, we deduce an isomorphism that sends the supercharacter basis of $\scf(\fkut)$ to the fundamental basis of $\FQSym$.  

In Section \ref{PermutationCharacterStructure}, we examine the structure constants for a third canonical basis that arises in the normal lattice supercharacter theory construction.  As far as we know, this gives a new basis for $\FQSym$ that nevertheless has a nice combinatorial structure.   We use the representation theoretic functors to compute the coproduct in Theorem \ref{PermutationCharacterCoproduct}, and we believe a more combinatorial approach would be significantly more complicated; here the coefficients are in the set $\{0,1\}$.  With slightly more effort Theorem \ref{PermutationCharacterProduct} computes the product in this basis, and the coefficients are in the set $\{-1,0,1\}$.  

While we categorify $\FQSym$, we do not supply much evidence that our construction is canonical; in fact, it seems likely that it is one of many possible choices for a tower of groups.  However, in Section 6 we construct a Hopf monoid that recovers $\FQSym$ via a Fock functor (as described in \cite{AM10}).  In any case, this paper should be viewed as more of a ``proof of concept" for the method outlined above.   To give a more robust connection one would ideally interpret all the relations between $\FQSym$ and other Hopf algebras via functors on the corresponding towers of groups, but at present this remains largely unexplored.  

\vspace{.5cm}

\noindent\textbf{Acknowledgements.}
The second author was supported by Simons Foundation collaboration grant 426594.  We would also like to thank an anonymous referee for detailed and thoughtful comments that  improved the paper.

\section{Preliminaries}

In this section we set up our notation for permutation combinatorics and introduce the Malvenuto--Reutenauer Hopf algebra.  We then review supercharacter theory fundamentals.

\subsection{Permutations}\label{Permutations}

For $n\in \ZZ_{\geq 0}$, let $S_n$ be the symmetric group on the set $\{1,2,\ldots, n\}$.  In this paper, we will use a number of different ways to represent elements of this group (see also Stanley \cite{StV1}).  
\begin{description}
\item[One line notation.]  For $w\in S_n$, we write the anagram of $12\cdots n$ given by $w(1)w(2)\cdots w(n)$.  
\item[Inversion table.]   The \emph{inversion table} of $w\in S_n$ is the sequence $\iota(w)=(\iota_1(w),\iota_2(w),\ldots,\iota_n(w))$, where
$$\iota_k(w)=\#\{i<w^{-1}(k)\mid w(i)>k\}.$$
In one line notation of $w$, $\iota_k(w)$ is the number of integers to the left of $k$ that are bigger than $k$.
For example, $\iota(314625)=(1,3,0,0,1,0)$.  Note that for each $k$, $0\leq \iota_k(w)\leq n-k$.  In fact,
\begin{equation}\label{PermutationToTable}
\begin{array}{r@{\ }ccc}
\iota:& S_n & \longrightarrow & \{(\alpha_1,\ldots,\alpha_n)\in \ZZ^n\mid 0\leq \alpha_k\leq n-k, 1\leq k\leq n\}\\
& w & \mapsto & \iota(w)
\end{array}
\end{equation}
is a bijection.
%\item[Code.]   The \emph{code} of $w\in S_n$ is the sequence $\kappa(w)=(\kappa_1(w),\kappa_2(w),\ldots,\kappa_n(w))$, where
%$$\kappa_k(w)=\#\{i<w(k)\mid w^{-1}(i)>k\}.$$
%In one line notation of $w$, $\kappa_k(w)$ is the number of integers to the right of $w(k)$ that are less than $w(k)$.
%In our example, $\kappa(314625)=(2,0,1,2,0,0)$.  Note that for each $k$, $0\leq \kappa_k(w)\leq n-k$ and $\kappa(w)=\iota(w^{-1})$.
\item[Rothe diagram.] The \emph{Rothe diagram} of $w\in S_n$ is the subset
$$R_w=\{(i,j)\mid w(i)>j, w^{-1}(j)>i\}\subseteq \{(i,j)\mid 1\leq i,j\leq n\}.$$
In our running example, the $R_{314625}$ is the set of coordinates marked by $\circ$ in the decorated matrix
$$\left[
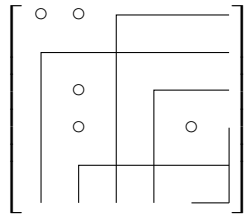
\begin{tikzpicture}[scale=.5,baseline=1.65cm]
\foreach \i/\j in {1/3,2/1,3/4,4/6,5/2,6/5}
	\draw (\j,1) -- (\j,7-\i) -- (6,7-\i); 
\foreach \i/\j in {1/1,1/2,3/2,4/2,4/5}
	\node at (\j,7-\i) {$\circ$};
\end{tikzpicture}\right]$$
Note that $\iota_k(w)=\#\{i\mid (i,k)\in R_w\}$.  Also useful below is that $\iota_k(w^{-1})=\#\{j\mid (k,j)\in R_w\}$.
%\quad \text{and}\quad \kappa_k(w)=\#\{(k,j)\in R_w\}.$
\end{description}
There is a natural poset on $\ZZ_{\geq 0}^n$ given by
\begin{equation}\label{PermutationSequenceOrder}
v\geq u \qquad \text{if and only if} \qquad \text{$v_i\geq u_i$ for  all $1\leq i\leq n$}.
\end{equation}
Applying this poset to inversion tables gives a corresponding distributive lattice on permutations. 
%This poset defines two distributive lattices on permutations: one applies the order to inversion tables, and one applies the order to codes. 
For example, with $S_4$, we obtain
$$
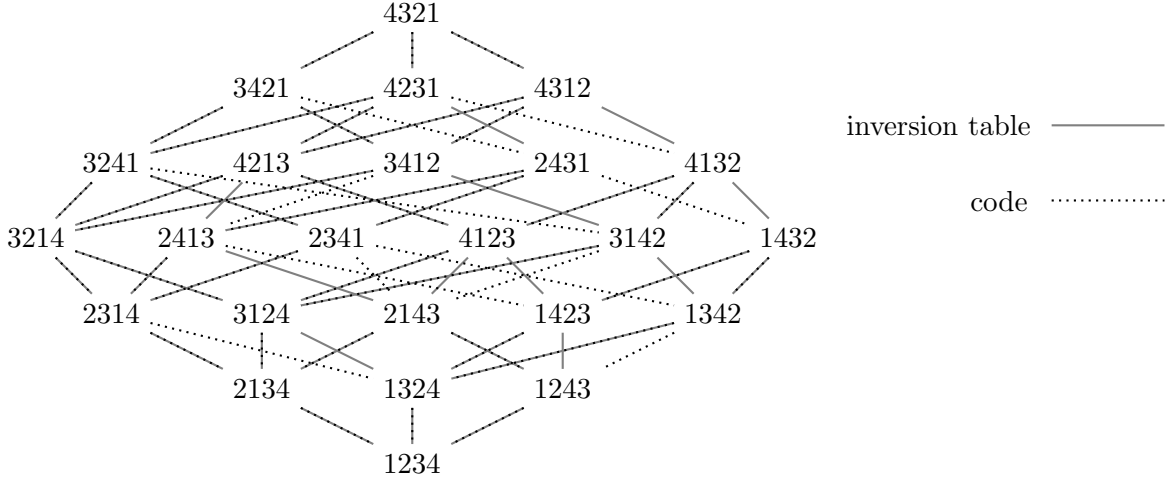
\begin{tikzpicture}[baseline=2cm]
\foreach \x/\y/\z in {0/6/4321,
-2/5/3421,0/5/4231,2/5/4312,
-4/4/3241,-2/4/4213,0/4/3412,2/4/2431,4/4/4132,
-5/3/3214,-3/3/2413,-1/3/2341, 1/3/4123, 3/3/3142, 5/3/1432, 
-4/2/2314,-2/2/3124,0/2/2143,2/2/1423,4/2/1342,
-2/1/2134,0/1/1324,2/1/1243,0/0/1234,
}
\node (\z) at (\x,\y) {$\z$};
\foreach \a/\b in {2134/1234,1324/1234,1243/1234,3124/2134,2314/2134,2143/2134,3124/1324,2143/1243,1342/1324,1423/1324,1423/1243,2341/2314,3214/2314,3214/3124,3142/3124,3142/1342,2413/2314,2413/2143,4123/3124,4123/2143,4123/1423,1432/1342,1432/1423,3241/3214,3241/2341,2431/2341,2431/2413,3412/3214,3412/3142,4213/2413,4213/4123,4213/3214,4132/1432,4132/4123,4132/3142,4231/4213,4231/2431,4231/3241,3421/3412,3421/3241,4312/4132,4312/4213,4312/3412,4321/4231,4321/3421,4321/4312}
\draw[thick,gray] (\a) -- (\b);
%\foreach \a/\b in {2134/1234,1324/1234,1243/1234,2314/2134,2314/1324,3124/2134,2143/2134,2143/1243,1342/1324,1342/1243,1423/1324,3214/2314,3214/3124,2341/2314,2341/2143,2341/1342,4123/3124,2413/2314,2413/1423,3142/3124,3142/2143,1432/1342,1432/1423,3241/3214,3241/2341,3241/3142,4213/3214,4213/4123,3412/3214,3412/2413,2431/2341,2431/2413,2431/1432,4132/4123,4132/3142,3421/3241,3421/3412,3421/2431,4312/4213,4312/3412,4231/3241,4231/4213,4231/4132,4321/3421,4321/4312,4321/4231}
%\draw[dotted,thick] (\a) -- (\b);
%\node at (7,4.5) {inversion table};
%\node at (7.8,3.5) {code};
%\draw[thick,gray] (8.5,4.5) -- (10,4.5);
%\draw[dotted,thick]  (8.5,3.5) -- (10,3.5);
\end{tikzpicture}\ .$$

\subsection{The Malvenuto--Reutenauer Hopf algebra $\FQSym$}

The Malvenuto--Reutenauer algebra is a graded Hopf algebra with underlying vector space
$$\FQSym=\bigoplus_{n\in \ZZ_{\geq 0}} \CC\spanning\{F_w\mid w\in S_n\}.$$
To define an algebra structure on $\FQSym$, we define a notion of shifted shuffle.  Given $v\in S_m$, $w\in S_n$ and $A\subseteq\{1,2,\ldots,m+n\}$ with $|A|=n$, define the \emph{$A$-shuffle} $v\shuffle_A w\in S_{m+n}$ by
$$(v\shuffle_A w)(i)=\left\{\begin{array}{ll} v(i-\#\{a\in A\mid a<i\}) & \text{if $i\notin A$}\\
w(\#\{ a\in A\mid a\leq i\})+m & \text{if $i\in A$.}
\end{array}\right.$$
For example, 
$$31542\shuffle_{\{1,4,5,8\}} 3124 =
\begin{tikzpicture}[baseline=-.1cm]
\foreach \x/\y in {0/3,1/1,2/5,3/4,4/2}
	\node (\x) at (\x/4,0) {$\y$};
\foreach \x in {6,...,14}
	\node (\x) at (\x/4,0) {$-$};
\foreach \x/\y in {6/1,9/4,10/5,13/8}
	\node[gray] at (\x/4,.25) {$\scs\y$};
\foreach \x/\y in {16/3,17/1,18/2,19/4}
	\node (\x) at (\x/4,0) {$\y$};
\foreach \s/\t in {0/7,1/8,2/11,3/12,4/14}
	\draw[->] (\s) to  [in=90, out=90] (\t); 
\foreach \s/\t in {16/6,17/9,18/10,19/13}
	\draw[->] (\s) to  [in=-90, out=-90] (\t); 
\end{tikzpicture}
= \overset{{\color{gray} 1}}{8}31\overset{{\color{gray}4}}{6}\overset{{\color{gray}5}}{7}54\overset{{\color{gray}8}}{9}2.$$
Define
$$v\shuffle w=\{v\shuffle_A w\mid A\subseteq\{1,2,\ldots, m+n\} \text{ with } |A|=n\}.$$
The product on $\FQSym$ is given by
\begin{equation}\label{FQSymProduct}
F_vF_w=\sum_{y\in v\shuffle w} F_y.
\end{equation}

To define a coalgebra structure on $\FQSym$, we define a standardized deconcatenation. For $w\in S_{m+n}$, the \emph{$m$-standardized deconcatenation} of $w$ is the pair $(w_{\leq m},w_{>m})\in S_m\times S_n$, where
\begin{align*}
w_{\leq m}(i) &=w(i)-\#\{j\mid m<j\leq m+n\text{ and }w(j)<w(i) \}\\
w_{>m}(j) & =w(j+m)-\#\{i\mid 1\leq i\leq m\text{ and } w(i)<w(j) \}. 
\end{align*}

For example, the $5$-standardized deconcatenation of $319825647$ is 
$$\begin{tikzpicture}[baseline=.5cm]
\foreach \x/\y in {0/3,1/1,2/9,3/8,4/2,5/5,6/6,7/4,8/7}
	\node (\x1) at (\x/4,1) {$\y$};
\foreach \x/\y in {0/3,1/1,2/5,3/4,4/2}
	\node (\x0) at (\x/4-.5,0) {$\y$};
\foreach \x/\y in {5/2,6/3,7/1,8/4}
	\node (\x0) at (\x/4+.5,0) {$\y$};
\foreach \s in {0,...,8}
	\draw[->] (\s1) -- (\s0); 
\end{tikzpicture}.$$
The coproduct on $\FQSym$ is given by 
\begin{equation}\label{FQSymCoProduct}
\Delta(F_w)=\sum_{m=0}^n F_{w_{\leq m}}\otimes F_{w_{>m}}\quad \text{for $w\in S_n$}.
\end{equation}

\subsection{Supercharacter theories}\label{SupercharacterTheories}

Supercharacter theories were introduced by \cite{DI} as a means to get representation theoretic control of groups with difficult representation theories (e.g. the Sylow $p$-subgroups of the finite general linear groups $\GL_n(\FF_p)$).  However, \cite{FGK} also showed that one may use these theories to make the representation theory of less exciting groups (e.g. abelian groups) more compelling.  The examples of supercharacter theories used in this paper will be for abelian groups, so we will denote the group operation by $+$ and the identity by $0$.   Essentially, a supercharacter theory constructs a well-behaved subspace of the vector space of class functions
$$\cf(G)=\{\psi:G\rightarrow \CC\mid \psi(hgh^{-1})=\psi(g),g,h\in G\}$$
which in our abelian case is in fact the space of all  functions $G\rightarrow \CC$.

A \emph{supercharacter theory} $(\Cl,\Ch)$ of a finite abelian group $G$ is a pair of partitions where $\Cl$ is a partition of the group and $\Ch$ is a partition of the irreducible characters $\Irr(G)$ such that 
\begin{enumerate}
\item[(SC1)] $\{0\}\in \Cl$,
\item[(SC2)] $|\Cl|=|\Ch|$,
\item[(SC3)] For each $A\in \Ch$, $\sum_{\psi\in A}\psi(g)=\sum_{\psi\in A}\psi(h)$ whenever $g,h\in K$ for some $K\in \Cl$.
\end{enumerate}  
This definition is equivalent to the one given in \cite{DI}, though specialized to abelian groups.  We typically call the blocks of $\Cl$ \emph{superclasses}.
In fact, condition (SC2) and (SC3) imply that the subspace of functions
$$\scf(G)=\{\gamma:G\rightarrow \CC\mid \gamma(g)=\gamma(h)\text{ whenever $g$ and $h$ are in the same superclass}\}\subseteq \cf(G)$$
has two distinguished bases:
\begin{description}
\item[Superclass identifier functions.]  For each $K\in \Cl$, define
$$\delta_K(g)=\left\{\begin{array}{ll} 1 & \text{if $g\in K$,}\\ 0 & \text{otherwise.}\end{array}\right.$$
\item[Supercharacters.] For each $A\in \Ch$, the corresponding \emph{supercharacter} is the function
$$\chi^A=\sum_{\psi\in A}\psi.$$
\end{description}

A given finite group typically has many supercharacter theories, but we will focus on a construction developed in \cite{Al} that works particularly well for groups with many normal subgroups (e.g. noncyclic abelian groups).   In the abelian context, we will refer to such a theory as a \emph{subgroup lattice} supercharacter theory.

\begin{theorem}[{\cite[Theorem 3.4]{Al}}]
Let $\cL$ be a set of subgroups of a finite abelian group $G$ containing both $G$ and $\{0\}$ such that for $M, N\in \cL$, we have $M\cap N, M+N\in \cL$.  
\begin{enumerate}
\item[(a)] Let $\Cl$ be the partition of $G$ obtained by placing $g,h\in G$ in the same block if and only if the smallest  subgroup in $\cL$ containing $g$ is also the smallest one containing $h$.
\item[(b)] Let $\Ch$ be the partition of $\Irr(G)$ obtained by placing $\psi,\tau\in \Irr(G)$ in the same block if and only if the largest  subgroup in $\cL$ contained in the kernel of $\psi$ is also the largest subgroup contained in the kernel of $\tau$.
\end{enumerate}
Then $(\Cl,\Ch)$ is a supercharacter theory of $G$.
\end{theorem}

A feature of this result is that (a) associates a subgroup $N$ to each superclass (and we will refer to the corresponding identifier function as $\delta_N$), and (b) associates a  subgroup  $N$ to each supercharacter (referred to as $\chi^N$).   For $N\in \cL$, we get some additional class functions 
$$\bar\delta_N=\sum_{M\subseteq N}\delta_M\quad \text{and} \quad \bar\chi^N=\sum_{O\supseteq N} \chi^O.$$
In fact,  if $\One$ is the trivial character of $G$, then
\begin{equation}\label{PermutationCharacterBasis}
\bar\delta_N(g)=\left\{\begin{array}{ll} 1 & \text{if $g\in N$,}\\ 0 & \text{otherwise,}\end{array}\right.\qquad\text{and} \qquad 
\bar\chi^N=\Ind_N^G(\One),
\end{equation}
so $\bar\delta_N=\frac{|N|}{|G|}\bar\chi^N$.  Therefore these constructions effectively add a third canonical basis (up to scaling) for $\scf(G)$.

\section{The vector space $\scf(\fkut_n)$}

In this section, we construct the main vector space $\scf(\fkut_n)$ using a sublattice of subgroups of $\fkut_n$ to give a subgroup lattice supercharacter theory.  We then introduce a somewhat mysterious involution on this space that will be important for the Hopf algebra structure.

\subsection{A subgroup lattice theory for $\fkut_n$} 
Fix a finite field $\FF_q$ with $q$ elements and let $M_n(\FF_q)$ denote the algebra of $n\times n$ matrices with entries in $\FF_q$.  Define the nilpotent subalgebra
$$\fkut_n=\{x\in M_n(\FF_q)\mid x_{ij}\neq 0\text{ implies } i<j\}.$$
If $q$ is a prime $p$-power, then the additive group of $\fkut_n$ is an elementary abelian $p$-group.   For each $w\in S_n$, define
\begin{equation}
\fkut_w =\{x\in \fkut_n\mid x_{ij}\neq 0 \text{ implies } 0< j-i\leq \iota_i(w)\}.
\end{equation}
For example, if $w=314625$, then $\iota(w)=(1,3,0,0,1,0)$ and 
$$\fkut_w=\left[\begin{array}{cccccc} 
0 & * & 0 & 0 & 0 & 0\\
0 & 0 & * & * & * & 0 \\
0 & 0 & 0 & 0 & 0 & 0 \\
0 & 0 & 0 & 0 & 0 & 0  \\
0 & 0 & 0 & 0 & 0 & * \\
0 & 0 & 0 & 0 & 0 & 0\\
\end{array}\right]\subseteq 
\left[\begin{array}{cccccc} 
0 & * & * & * & * & *\\
0 & 0 & * & * & * & * \\
0 & 0 & 0 & * & * & * \\
0 & 0 & 0 & 0 & * & *  \\
0 & 0 & 0 & 0 & 0 & * \\
0 & 0 & 0 & 0 & 0 & 0\\
\end{array}\right]=\fkut_6.$$
Subgroup containment gives a lattice
 \begin{equation}\label{SubgroupSet}
 \cL_n=\{\fkut_w\mid w\in S_n\}
 \end{equation}
isomorphic to the inversion table order (\ref{PermutationSequenceOrder}) on $S_n$. In fact, if $\alpha$ and $\beta$ are inversion tables, we may use the inverse of the bijection (\ref{PermutationToTable}) to get
\begin{align*}
\fkut_{\iota^{-1}(\alpha)}\cap \fkut_{\iota^{-1}(\beta)} & =\fkut_{\iota^{-1}(\min(\alpha,\beta))},\quad \text{where} \quad \min(\alpha,\beta)_k=\min(\alpha_k,\beta_k),\\
\fkut_{\iota^{-1}(\alpha)}+\fkut_{\iota^{-1}(\beta)}  & = \fkut_{\iota^{-1}(\max(\alpha,\beta))},\quad \text{where} \quad \max(\alpha,\beta)_k=\max(\alpha_k,\beta_k).
\end{align*}
  The following result re-interprets the covers in the inversion table order directly on permutations.

\begin{proposition}\label{CoveringInversions}
The permutation $w\in S_n$ covers $v\in S_n$ in the inversion table order if and only if there exists $i<k$ such that 
\begin{enumerate}
\item[(1)] $w(j)=v(j)$ for all $j\notin \{i,k\},$
\item[(2)] $v(i)=w(k)<w(i)=v(k)$,
\item[(3)] for each $i<j<k$, $w(j)<w(k)$.
\end{enumerate}
\end{proposition}
\begin{proof}
By definition $w$ covers $v$ if and only if there exists $b$ such that $\iota_b(w)=\iota_b(v)+1$ and $\iota_c(w)=\iota_c(v)$ for all $c\neq b$.  

It is straight-forward to check that (1), (2), and (3) imply that, in this case, such a $b$ is given by $b=w(k)$.  Conversely, suppose there exists $b=w(k)$ such that $\iota_b(w)=\iota_b(v)+1$ and $\iota_c(w)=\iota_c(v)$ for all $c\neq b$.  Let $i<k$ be maximal such that $w(i)>w(k)$ (which must exist by assumption).  Then
$$v'=w(1)w(2)\cdots w(i-1) w(k)w(i+1)\cdots w(k-1)w(i)w(k+1)\cdots w(n)$$
satisfies $\iota_b(w)=\iota_b(v')+1$ and $\iota_c(w)=\iota_c(v')$ for all $c\neq b$. Since $\iota$ is bijective, $v'=v$.
\end{proof}

Let $(\Cl_n,\Ch_n)$ be the subgroup lattice supercharacter theory of $\fkut_n$ associated with the lattice $\cL_n$ (\ref{SubgroupSet}).
The superclasses are given by
\begin{equation}\label{Superclasses}
\Cl_w=\fkut_w-\bigcup_{\fkut_v\subset \fkut_w} \fkut_v,
\end{equation}
which is a nonempty set for each $w\in S_n$.  As in Section \ref{SupercharacterTheories}, the vector space of superclass functions
$$\scf(\fkut_n)=\{\psi:\fkut_n\rightarrow \CC\mid \psi(x)=\psi(y),  \text{ if $x,y\in \Cl_w$ for some $w\in S_n$}\}$$
has three distinguished bases
$$\{\delta_w\mid w\in S_n\},\quad \{\bar\chi^w\mid w\in S_n\},\quad \text{and}\quad \{\chi^w\mid w\in S_n\},$$
where it is notationally convenient to label each basis element with the underlying permutation $w$ rather than the normal subgroup $\fkut_w$. From \cite[Corollary 3.4]{AT}, we  obtain a supercharacter formula.

\begin{proposition}
For $w,v\in S_n$ and $x\in \Cl_w$,
$$\chi^v(x)=\left\{\begin{array}{@{}ll}
\frac{|\fkut_n|}{|\fkut_v|}\left(\frac{q-1}{q}\right)^{\#\{j\mid \iota_j(v)<n-j\}}\left(\frac{-1}{q-1}\right)^{\#\{j\mid \iota_j(w)=\iota_j(v)+1\}} & \text{if $\iota_j(w)\leq \iota_j(v)+1$ for all $1\leq j\leq n$,}\\ 0 & \text{otherwise.}
\end{array}\right.$$
\end{proposition}

We may re-interpret this result as an entry by entry factorization.  Let $\One_{\FF_q^+}$ be the trivial character of $\FF_q^+$ and $\reg_{\FF_q^+}$ be the regular character of $\FF_q^+$.

\begin{corollary}\label{SupercharacterFactorization}
For $v\in S_n$ and $x\in \fkut_n$,
$$\chi^v(x)=\prod_{1\leq i<j\leq n\atop j-i\leq \iota_i(v)} \One_{\FF_q^+}(x_{ij}) \prod_{1\leq i<j\leq n\atop j-i= \iota_i(v)+1} (\reg_{\FF_q^+}-\One_{\FF_q^+})(x_{ij}) \prod_{1\leq i<j\leq n\atop j-i> \iota_i(v)+1}\reg_{\FF_q^+}(x_{ij}).
$$
\end{corollary}

In light of this corollary, and for the purpose of examples below, the following notation can be helpful.  Let $\varphi\in M_n\Big(\cf(\FF_q^+)\Big)$ be a strictly upper-triangular $n\times n$ matrix with functions from $\FF_q^+\rightarrow \CC$ as entries.  Then for $x\in \fkut_n$ we can define
\begin{equation} \label{MatrixFunctions}
\varphi(x)=\prod_{i<j} \varphi_{ij}(x_{ij}).
\end{equation}

\begin{example}  In light of Corollary \ref{SupercharacterFactorization}, for example,
\begin{equation} \label{SupercharacterMatrixExample}
\chi^{5317426}=\chi^{\iota^{-1}(2,4,1,2,0,1,0)}=
\left[\begin{smallmatrix}
\cdot & \One & \One & \reg-\One & \reg & \reg & \reg\\
 & \cdot & \One & \One & \One & \One &  \reg-\One \\
 & & \cdot & \One & \reg-\One & \reg& \reg \\
&  & & \cdot & \One & \One & \reg-\One \\
&  & & & \cdot & \reg-\One & \reg  \\
&  & & & & \cdot & \One \\
&  & & & & & \cdot  
\end{smallmatrix}\right],
\end{equation}
where $\cdot$ denotes a diagonal entry.
\end{example}

Fix a nontrivial homomorphism $\vartheta:\FF_q^+\rightarrow \CC^\times$.  Then for each $x\in \fkut_n$, define a class function $\vartheta_x$ in the notation (\ref{MatrixFunctions}) given by
\begin{equation}\label{ThetaNotation}
(\vartheta_x)_{ij}=\vartheta_{x_{ij}}\qquad\text{where for $r,t\in \FF_q$, $\vartheta_r(t)=\vartheta(rt)$.}
\end{equation}
In this notation,
\begin{equation}\label{SupercharacterClassFunctions}
 \chi^w=\sum_{{x\in \fkut_n\atop x_{jk}=0, k-j\leq \iota_j(w)}\atop x_{j,j+\iota_j(w)+1}\neq 0}\vartheta_x\in \cf(\fkut_n).
 \end{equation}
The Borel subgroup
$$B_n=\{b\in \GL_n(\FF_q)\mid b_{ji}=0, i<j\}$$ 
of invertible upper-triangular matrices has a right action on $\fkut_n$, and we let it do so row by row.  In particular, for $x\in \fkut_n$ and $(b_1,\ldots, b_n)\in B_n^n$, define
$$x(b_1,\ldots, b_n)=\left[\begin{array}{c} x_{1\ast} b_1\\ x_{2\ast} b_2\\ \vdots \\ x_{n\ast} b_n\end{array}\right], \quad \text{where $x_{i\ast}$ is the $i$th row of $x$}.$$
By (\ref{SupercharacterClassFunctions}), the function
\begin{equation} \label{SymMap}\begin{array}{rccc} \pi_B: & \cf(\fkut_n) & \longrightarrow & \scf(\fkut_n)\\ 
& \vartheta_x & \mapsto & \dd\frac{1}{|B_n|^n}\sum_{\underline{b}\in B_n^n} \vartheta_{x\underline{b}}, 
\end{array}
\end{equation}
is a surjective projection.

\subsection{A $\star$-duality}

The linear bijection
$$\begin{array}{rccccc} \star:  & \scf(\fkut_n) & \longrightarrow & \scf(\fkut_n) & \longrightarrow & \scf(\fkut_n)\\
& \dd\sum_{w\in S_n} c_w\chi^w & \mapsto & \dd\sum_{w\in S_n} c_w \chi^{w^{-1}}  & \mapsto & \dd\sum_{w\in S_n} c_w \frac{\chi^{w^{-1}}}{\chi^{w^{-1}}(0)}\end{array}$$
is the composition of a combinatorial involution (inversion of the permutation) with a representation theoretic bijection (duality with respect to the supercharacter basis and the usual inner product on class functions).  It is worth noting several representation theoretic stability properties of the combinatorial inversion.

\begin{proposition} For $w\in S_n$,
$$\fkut_n/\fkut_{w^{-1}}\cong \fkut_n/\fkut_w$$
\end{proposition}
\begin{proof}
We claim that $n-j-\iota_j(w^{-1})=n-w(j)-\iota_{w(j)}(w)$ for all $1\leq j\leq n$.  Note that
\begin{align*}
n-j-\iota_j(w^{-1}) &=n-j-\#\{i\mid i<w(j)\text{ and } w^{-1}(i)>j\}\\
&=\#\{i\mid i>w(j)\text{ and }  w^{-1}(i)>j\}\\
&=n-w(j)-\#\{i\mid i<j\text{ and } w(i)>w(j)\}\\
&=n-w(j)-\iota_{w(j)}(w).
\end{align*}  
The result now follows from the definition of $\fkut_w$ and its implicit complement in $\fkut_n$.  
\end{proof}
\begin{remark}
The proof of the proposition in fact shows that inverting $w$ permutes the entries in the vector 
\begin{equation} \label{DualInversionTable}
\iota^\vee(w)=(n-1-\iota_1(w),n-2-\iota_2(w),\ldots, n-n-\iota_n(w)).
\end{equation}
Thus,  $\chi^{w^{-1}}(0)=\chi^w(0)$, so in the $\star$-function it does not matter in which order we apply the two bijections.   On the other hand, while $|\fkut_w|=|\fkut_{w^{-1}}|$, inversion does not preserve superclass size.  For example,  
$$|\fkut_{(312)}|=\left|\left[\begin{array}{ccc} 0 & \ast & 0\\ 0 & 0 & \ast\\ 0 & 0 &0\end{array}\right]\right|=\left|\left[\begin{array}{ccc} 0 & \ast & \ast\\ 0 & 0 & 0\\ 0 & 0 &0\end{array}\right]\right|=|\fkut_{(231)}|=q^2,$$ 
but by (\ref{Superclasses}) the sizes of the corresponding superclasses are $(q-1)^2$ and $q(q-1)$, respectively.
\end{remark}

\section{The Hopf algebra $\scf(\fkut)$} \label{HopfAlgebraIsomorphism}

The goal of this section is to define functors that give a Hopf structure to the space
$$\scf(\fkut)=\bigoplus_{n\geq 0}\scf(\fkut_n).$$
  By computing the structure constants on the supercharacter basis, we deduce that the Hopf algebra is in fact isomorphic to the Malvenuto--Reutenauer Hopf algebra $\FQSym$. 

\subsection{Functorial subgroups} \label{SectionSubgroups}

We begin by defining a family of subgroups of $\fkut_\ell$ that depend on a subset $A\subseteq I_\ell=\{1,2,\ldots, \ell\}$.   First define a partition of $\{1\leq i<j\leq \ell\}$ with blocks
\begin{align*}
U_A&=\{(i,j)\in A\times I_\ell\mid i< j,j> \ell-\#\{a\in A\mid a>i\}\}\\
L_A&=\{(i,j)\in A\times I_\ell\mid  i< j\leq \ell-\#\{a\in A\mid a>i\}\}\\
U_A^\vee&=\{(i,j)\in (I_\ell-A)\times I_\ell \mid  i< j\leq \ell-\#\{a\in A\mid a>i\}\}\\
R_A&=\{(i,j)\in (I_\ell-A)\times I_\ell \mid  i<j,j> \ell-\#\{a\in A\mid a>i\}\}.
\end{align*}
For example, if $A=\{1,4,5, 7\}\subseteq \{1,2,\ldots, 9\}$, then
$$
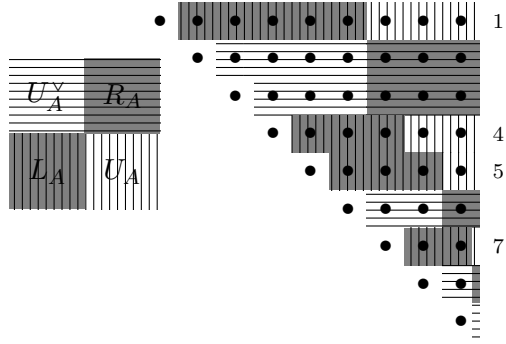
\begin{tikzpicture}[scale=.5]
	\fill[pattern=vertical lines] (7,8) +(-1.5,-0.5) rectangle ++(1.5,0.5);
	\fill[gray] (3,8)  +(-2.5,-0.5) rectangle ++(2.5,0.5);
	\fill[pattern=vertical lines] (3,8)  +(-2.5,-0.5) rectangle ++(2.5,0.5);
	\fill[pattern=horizontal lines] (3.5,7) +(-2,-0.5) rectangle ++(2,0.5);
	\fill[gray] (7,7) +(-1.5,-0.5) rectangle ++(1.5,0.5);
	\fill[pattern=horizontal lines] (7,7) +(-1.5,-0.5) rectangle ++(1.5,0.5);
	\fill[pattern=horizontal lines] (4,6) +(-1.5,-0.5) rectangle ++(1.5,0.5);
	\fill[gray] (7,6) +(-1.5,-0.5) rectangle ++(1.5,0.5);
	\fill[pattern=horizontal lines] (7,6) +(-1.5,-0.5) rectangle ++(1.5,0.5);
	\fill[pattern=vertical lines] (7.5,5) +(-1,-.5) rectangle ++(1,.5);
	\fill[gray] (5,5) +(-1.5,-0.5) rectangle ++(1.5,0.5);
	\fill[pattern=vertical lines] (5,5) +(-1.5,-0.5) rectangle ++(1.5,0.5);
	\fill[pattern=vertical lines] (8,4) +(-.5,-.5) rectangle ++(.5,.5);
	\fill[gray] (6,4) +(-1.5,-0.5) rectangle ++(1.5,0.5);
	\fill[pattern=vertical lines] (6,4) +(-1.5,-0.5) rectangle ++(1.5,0.5);
	\fill[pattern=horizontal lines] (6.5,3) +(-1,-0.5) rectangle ++(1,0.5);
	\fill[gray] (8,3) +(-.5,-0.5) rectangle ++(.5,0.5);
	\fill[pattern=horizontal lines] (8,3) +(-.5,-0.5) rectangle ++(.5,0.5);
	\fill[pattern=vertical lines] (8.4,2) +(-.1,-.5) rectangle ++(.1,.5);
	\fill[gray] (7.5,2) +(-1,-0.5) rectangle ++(.8,0.5);
	\fill[pattern=vertical lines] (7.5,2) +(-1,-0.5) rectangle ++(.8,0.5);
	\fill[pattern=horizontal lines] (8,1) +(-.5,-0.5) rectangle ++(.3,0.5);
	\fill[gray] (8.4,1) +(-.1,-0.5) rectangle ++(.1,0.5);
	\fill[pattern=horizontal lines] (8.4,1) +(-.1,-0.5) rectangle ++(.1,0.5);
	\fill[pattern=horizontal lines] (8.4,0) +(-.1,-0.5) rectangle ++(.1,0.5);
\foreach \x in {0,...,8}
	{\foreach \y in {\x,...,8}
	\node (\x,\y) at (8-\x,\y) {$\bullet$};}
\foreach \y/\e in {2/7,4/5,5/4,8/1}
	\node at (9,\y) {$\scriptstyle\e$};
\fill[pattern=vertical lines] (-1,4) +(-1,-1) rectangle ++(1,1);
\node at (-1,4) {$U_A$};
\fill[gray] (-3,4) +(-1,-1) rectangle ++(1,1);
\fill[pattern=vertical lines] (-3,4) +(-1,-1) rectangle ++(1,1);
\node at (-3,4) {$L_A$};
\fill[pattern=horizontal lines] (-3,6) +(-1,-1) rectangle ++(1,1);
\node at (-3,6) {$U_A^\vee$};
\fill[gray] (-1,6) +(-1,-1) rectangle ++(1,1);
\fill[pattern=horizontal lines] (-1,6) +(-1,-1) rectangle ++(1,1);
\node at (-1,6) {$R_A$};
\end{tikzpicture}$$

These coordinates give rise to subgroups
\begin{align*}
\fkut_A &= \{u\in \fkut_n\mid u_{ij}\neq 0 \text{ implies } (i,j)\in U_A\}\\
\fkl_A & = \{u\in \fkut_n\mid u_{ij}\neq 0\text{ implies }  (i,j)\in L_A\}\\
\fkut_A^\vee &= \{u\in \fkut_n\mid u_{ij}\neq 0\text{ implies } (i,j)\in U_A^\vee\}\\
\fkr_A & =  \{u\in \fkut_n\mid u_{ij}\neq 0 \text{ implies } (i,j)\in R_A\}\\
\fkp_A &= \fkl_A+\fkut_A^\vee +\fkut_A.
\end{align*}
Note that $\fkut_\ell= \fkp_A+\fkr_A$, $\fkut_{|A|}\cong \fkut_A$ and $\fkut_{\ell-|A|}\cong \fkut_A^\vee$.  In fact, if $\ell=m+n$ with $|A|=n$, consider the explicit isomorphisms
\begin{equation} \label{CheckIsomorphism}
\begin{array}{rccc} \tau'_A: & \fkut_A^\vee & \longrightarrow & \fkut_{m}\\
& e_{ij}(t) & \mapsto & e_{i-\#\{a\in A\mid a<i\},j-\#\{a\in A\mid a<i\}}(t),\end{array}
\end{equation}
and 
\begin{equation}\label{UnCheckIsomorphism}
\begin{array}{rccc} \tau_A: & \fkut_A & \longrightarrow & \fkut_{n}\\
& e_{ij}(t) & \mapsto & e_{i-\#\{b\in I_\ell-A\mid b<i\},j-m}(t),\end{array}
\end{equation}
where $e_{ij}(t)$ is the matrix with $t$ in the $i$th row and $j$th column and zeroes elsewhere.
Our functors, below, will pass up from $\fkut_A^\vee\times \fkut_A$ through $\fkp_A$ to $\fkut_\ell$ or down in reverse.

\subsection{The functor $\Sfl$ that goes up}

From this point on, we will use the constructions of Section \ref{SectionSubgroups}, with $\ell=m+n$ and $|A|=n$.  Let $A\subseteq \{1,2,\ldots,m+n\}$ with $|A|=n$.  The function $\Sfl$ will be the composition of two functors.  The first is inflation 
$$\begin{array}{rccc}
 \Inf_{\fkut_A^\vee\times \fkut_A}^{\fkp_A} : & \scf(\fkut_A^\vee\times \fkut_{A}) & \longrightarrow & \scf(\fkp_A)\\
& \chi\otimes \psi & \mapsto &\begin{array}{c@{\ }c@{\ }c} \fkp_A & \rightarrow & \CC\\ l+u'+u & \mapsto &  \chi(u')\psi(u)\end{array}\end{array} \quad \text{for $u'\in \fkut_A^\vee$, $l\in \fkl_A$, $u\in \fkut_A$.}$$
\begin{example}
Using the notation (\ref{MatrixFunctions}), we get
\begin{align*} \Inf_{\fkut_{\{1,3,6\}}^\vee\times \fkut_{\{1,3,6\}}}^{\fkp_{\{1,3,6\}}}(\chi^{2143}\otimes\chi^{312}) &= \Inf_{\fkut_{\{1,3,6\}}^\vee\times \fkut_{\{1,3,6\}}}^{\fkp_{\{1,3,6\}}}(\chi^{\iota^{-1}(1,0,1,0)}\otimes\chi^{\iota^{-1}(1,1,0)})\\ 
&= \Inf_{\fkut_{\{1,3,6\}}^\vee\times \fkut_{\{1,3,6\}}}^{\fkp_{\{1,3,6\}}}\biggl(
\left[\begin{smallmatrix}
\cdot & \One & \reg-\One & \reg\\ 
& \cdot & \One & \reg-\One \\
& & \cdot & \One\\
&  & & \cdot
\end{smallmatrix}\right]\otimes
\left[\begin{smallmatrix}
\cdot &  \color{gray}\One &  \color{gray}\reg-\One\\
& \cdot &  \color{gray}\One\\
&  & \cdot
\end{smallmatrix}\right]
\biggr)\\
&= \left[\begin{smallmatrix}
\cdot &  \One &  \One & \One & \One  & \color{gray}\One & \color{gray} \reg-\One\\
& \cdot & \One & \reg-\One & \reg & \bullet & \bullet \\ 
 & & \cdot & \One & \One & \One &  \color{gray} \One \\
& & & \cdot & \One & \reg-\One & \bullet \\
& & & & \cdot & \One & \bullet\\
& & &   & & \cdot & \One\\
& & & &   & & \cdot 
\end{smallmatrix}\right]\begin{smallmatrix}\color{gray} 1\\ \  \\  \color{gray} 3\\ \  \\ \ \\  \color{gray} 6\\ \ \end{smallmatrix}  ,
\end{align*}
where $\bullet$ denotes entries not in $\fkp_{\{1,3,6\}}$.
\end{example}

To obtain the second functor define the character $\chi^\faith_{\fkr_A}$ of $\fkr_A$ by
$$\chi^\faith_{\fkr_A}(r)=\prod_{i\notin A}\bigg( \prod_{j>i\text{ minimal}\atop (i,j)\in R_A } (\reg_{\FF_q^+}-\One)(r_{ij})\prod_{j>i\text{ not minimal}\atop (i,j)\in R_A} (\reg_{\FF_q^+})(r_{ij})\bigg).$$
If we take the supercharacter theory of $\fkr_A$ obtained by using the sublattice of $\cL_n$ consisting of subgroups that are in $\fkr_A$, then  $\chi^\faith_{\fkr_A}$ is the unique faithful supercharacter of that theory. 
\begin{example}
Continuing the example, we get
$$\chi^\faith_{\fkr_{\{1,3,6\}}} =\left[\begin{smallmatrix}
\cdot &   \bullet &  \bullet &  \bullet &  \bullet  &  \bullet & \bullet\\
& \cdot &  \bullet &  \bullet &  \bullet & \reg-\One & \reg  \\ 
 & & \cdot &  \bullet &  \bullet &  \bullet & \bullet \\
& & & \cdot &  \bullet &  \bullet & \reg-\One \\
& & & & \cdot &  \bullet &  \reg-\One\\
& & &   & & \cdot &  \bullet\\
& & & &   & & \cdot 
\end{smallmatrix}\right]\begin{smallmatrix} \color{gray} 1\\ \  \\  \color{gray} 3\\ \  \\ \ \\  \color{gray} 6\\ \ \end{smallmatrix},$$
where in this case $\bullet$ indicates the entries not in $\fkr_{\{1,3,6\}}$.
\end{example}

Define
$$\begin{array}{rccccc}  
& & & \cf(\fkut_{m+n}) \\
& & & \begin{tikzpicture} \node[rotate=90] at (0,0) {$\subseteq$};\end{tikzpicture}\\
\Str_{\fkp_A}^{\fkut_{m+n}}: & \cf(\fkp_A) & \longrightarrow & \cf(\fkp_A\times \fkr_A) & \longrightarrow  & \scf(\fkut_{m+n})\\
& \psi & \mapsto & \psi\otimes \dd\frac{\chi^\faith_{\fkr_A}}{\chi^\faith_{\fkr_A}(0)} & \mapsto &  \dd\pi_B\Big(\psi\otimes \frac{\chi^\faith_{\fkr_A}}{\chi^\faith_{\fkr_A}(0)}\Big), \end{array}$$
where $\pi_B$ is as in (\ref{SymMap}).  
\begin{example}
Continuing our example, 
\begin{align*}
\Str_{\fkp_{\{1,3,6\}}}^{\fkut_{7}}\left(\left[\begin{smallmatrix}
\cdot &  \One &  \One & \One & \One  & \One & \reg-\One\\
& \cdot & \One & \reg-\One & \reg & \bullet & \bullet \\ 
 & & \cdot & \One & \One & \One &\One \\
& & & \cdot & \One & \reg-\One & \bullet \\
& & & & \cdot & \One & \bullet\\
& & &   & & \cdot & \One\\
& & & &   & & \cdot 
\end{smallmatrix}\right]\begin{smallmatrix} \color{gray} 1\\ \  \\  \color{gray} 3\\ \  \\ \ \\  \color{gray} 6\\ \ \end{smallmatrix} \right)
&\mapsto \frac{1}{(q-1)^3q} \left[\begin{smallmatrix}
\cdot &  \One &  \One & \One & \One  & \One & \reg-\One\\
& \cdot & \One & \reg-\One & \reg & \reg-\One & \reg \\ 
 & & \cdot & \One & \One & \One & \One \\
& & & \cdot & \One & \reg-\One & \reg-\One \\
& & & & \cdot & \One & \reg-\One\\
& & &   & & \cdot & \One\\
& & & &   & & \cdot 
\end{smallmatrix}\right]\\
&\mapsto \frac{1}{|B_7^7|(q-1)^3q} \sum_{\underline{b}\in B_7^7} \left[\begin{smallmatrix}
\cdot &  \One &  \One & \One & \One  & \One & \reg-\One\\
& \cdot & \One & \reg-\One & \reg & \reg-\One & \reg \\ 
 & & \cdot & \One & \One & \One & \One \\
& & & \cdot & \One & \reg-\One & \reg-\One \\
& & & & \cdot & \One & \reg-\One\\
& & &   & & \cdot & \One\\
& & & &   & & \cdot 
\end{smallmatrix}\right]\cdot\underline{b}.
\end{align*}
Note that $\reg_{\FF_q}=\sum_{t\in \FF_q}\vartheta_t$ and $\reg-\One=\sum_{t\in \FF_q^\times}\vartheta_t$.   Let's focus on row 4 in the example, using the notation (\ref{ThetaNotation}).  Here, 
\begin{align} \sum_{b\in B_7} [\begin{array}{ccccccc} & & & \cdot & \One & \reg-\One &\reg-\One \end{array}] b &= \sum_{b\in B_7} [\begin{array}{ccccccc} & & & \cdot & \vartheta_0 & \sum_{s\in \FF_q^\times} \vartheta_s & \sum_{t\in \FF_q^\times} \vartheta_t\end{array}] b\label{ProjectionComputation}\\
 & = \sum_{b\in B_7\atop s,t\in \FF_q^\times} \vartheta_{[\begin{array}{ccccccc} & & & \cdot & 0 & s & t\end{array}]b} \notag \\
&=\sum_{b\in B_7\atop s,t\in \FF_q^\times} \vartheta_{[\begin{array}{ccccccc} & & & \cdot & 0 & b_{66}s & b_{77}t+b_{67}s\end{array}]} \notag\\
&=\frac{|B_7|(q-1)}{q}\sum_{s'\in \FF_q^\times\atop b_{67}'\in \FF_q} \vartheta_{[\begin{array}{ccccccc} & & & \cdot & 0 & s' & b_{67}'\end{array}]}\notag \\
&= \frac{|B_7|(q-1)}{q} [\begin{array}{ccccccc} & & & \cdot & \One & \reg-\One & \reg\end{array}]. \notag
\end{align}
By a similar computation for the other rows, we obtain
\begin{align*}\Str_{\fkp_{\{1,3,6\}}}^{\fkut_{7}}\left(\left[\begin{smallmatrix}
\cdot &  \One &  \One & \One & \One  & \One & \reg-\One\\
& \cdot & \One & \reg-\One & \reg & \bullet & \bullet \\ 
 & & \cdot & \One & \One & \One &\One \\
& & & \cdot & \One & \reg-\One & \bullet \\
& & & & \cdot & \One & \bullet\\
& & &   & & \cdot & \One\\
& & & &   & & \cdot 
\end{smallmatrix}\right] \right) &= \frac{1}{(q-1)q^3} \left[\begin{smallmatrix}
\cdot &  \One &  \One & \One & \One  & \One & \reg-\One\\
& \cdot & \One & \reg-\One & \reg & \reg & \reg \\ 
 & & \cdot & \One & \One & \One & \One \\
& & & \cdot & \One & \reg-\One & \reg \\
& & & & \cdot & \One & \reg-\One\\
& & &   & & \cdot & \One\\
& & & &   & & \cdot 
\end{smallmatrix}\right]\\
&=  \frac{1}{(q-1)q^3} \chi^{\iota^{-1}(5,1,4,1,1,1,0)}= \frac{1}{(q-1)q^3} \chi^{7245613}.
\end{align*}
\end{example}

Their composition is the \emph{stretchflation functor}
$$\Sfl_{\fkut_A^\vee\times \fkut_A}^{\fkut_{m+n}}=\Str_{\fkp_A}^{\fkut_{m+n}}\circ \Inf_{\fkut_A^\vee\times \fkut_A}^{\fkp_A}.$$

For $A\subseteq \{1,2,\ldots, m+n\}$ with $|A|=n$ and complement $\Ac$.  For $w\in S_m$ and $v\in S_n$, define the permutation $w\bowtie_A v\in S_{m+n}$ by
$$(w\bowtie_A v)^{-1}(j)=\left\{\begin{array}{ll} w^{-1}(\#\{i\in \Ac\mid i\leq j\}) & \text{if $j\in \Ac$,}\\ 
v^{-1}(\#\{a\in A\mid a\leq j\})+m & \text{if $j\in A$.}\end{array}\right.$$

\begin{lemma}\label{ProductLemma}
Let $A\subseteq \{1,2,\ldots, m+n\}$ with $|A|=n$ and complement $\Ac$.  Let $w\in S_m$ and $v\in S_n$.  
\begin{enumerate}
\item[(a)] The permutation $w\bowtie_A v$ satisfies
$$\iota_j(w\bowtie_A v)=\iota_{j-\#\{b\in A\mid b<j\}}(w)\quad \text{for $j\in \Ac$ and}\quad \iota^\vee_a(w\bowtie_A v)=\iota^\vee_{a-\#\{i\in\Ac\mid i<a\}}(v)\quad \text{for $a\in A$}.$$
\item[(b)] The character
$$\Sfl_{\fkut_A^\vee\times \fkut_A}^{\fkut_{m+n}}\Big(\frac{\chi^w}{\chi^w(0)}\otimes \frac{\chi^v}{\chi^v(0)}\Big)=\frac{\chi^{w\bowtie_A v}}{\chi^{w\bowtie_Av}(0)}.$$
\end{enumerate}
\end{lemma}

\begin{proof} 
(a) We show that the permutation obtained from the inversion table statistics give $w\bowtie_A v$.   Let $y\in S_k$, and $\iota(y)=(\iota_1(y),\ldots,\iota_k(y))$.  Then for $1\leq j\leq k$, we have
$$y^{-1}(j)=\begin{array}{l}\text{the $1+\iota_j(y)$th smallest element in increasing order of the set }\\  \{1,2,\ldots, k\}-\{y^{-1}(1),\ldots, y^{-1}(j-1)\}.\end{array}$$
Suppose $j\notin A$.  Then  $\iota_j(w\bowtie_A v)=\iota_{j-k}(w)$ where $k=\#\{a\in A\mid a<j\}$.  Thus,
\begin{equation}\label{NotInA}
(w\bowtie_A v)^{-1}(j)=\begin{array}{l} \text{the $1+\iota_{j-k}(w)$th element in increasing order of the set}\\ \{1,2,\ldots, m+n\}-\{(w\bowtie_A v)^{-1}(1),\ldots, (w\bowtie_A v)^{-1}(j-1)\}.\end{array}
\end{equation}
Similarly, if $j\in A$ and $k=\#\{i\in \Ac\mid i<j\}$, then
\begin{equation}\label{InA}
(w\bowtie_A v)^{-1}(j)=\begin{array}{l} \text{the $1+m-k+\iota_{j-k}(v)$th element in increasing order of the set}\\ \{1,2,\ldots, m+n\}-\{(w\bowtie_A v)^{-1}(1),\ldots, (w\bowtie_A v)^{-1}(j-1)\}.\end{array}
\end{equation}
If we show that for every $j\in A$, $(w\bowtie_A v)^{-1}(j)>m$ and for every $j\not\in A$, $(w\bowtie_A v)^{-1}(j)\leq m$, the desired result follows from \ref{NotInA} and \ref{InA}. Suppose these properties hold for every positive integer less than $j$. Then the set $\{(w\bowtie_A v)^{-1}(1),\ldots, (w\bowtie_A v)^{-1}(j-1)\}$ contains $|\{a\in A\mid a<j\}|$ elements bigger than $m$ and $|\{1\leq i<j\mid i\notin A\}|$ elements less than $m$. It now follows that if $j\in \Ac$, then 
 the $1+\iota_{j-k}(w)$th element in increasing order of the set $\{1,2,\ldots, m+n\}-\{(w\bowtie_A v)^{-1}(1),\ldots, (w\bowtie_A v)^{-1}(j-1)\}$ is less than or equal to $m$, and if $j\in A$,  the $1+m-k+\iota_{j-k}(v)$th element in increasing order of the set  $\{1,2,\ldots, m+n\}-\{(w\bowtie_A v)^{-1}(1),\ldots, (w\bowtie_A v)^{-1}(j-1)\}$ is greater than  $m$.
 
(b) This proof generalizes  the running example computation of this section.    Note that as a class function in $\cf(\fkut_{m+n})$,
$$\Inf_{\fkut_A^\vee\times \fkut_A}^{\fkp_A} \Big(\frac{\chi^w}{\chi^w(0)}\otimes \frac{\chi^v}{\chi^v(0)}\Big) \otimes \frac{\chi^\faith_{\fkr_A}}{\chi^\faith_{\fkr_A}(0)} =\frac{1}{\chi^w(0)\chi^v(0)\chi^\faith_{\fkr_A}(0)}\psi_{A}^{w,v},$$
where for $1\leq i<j\leq m+n$,
$$(\psi_{A}^{w,v})_{ij}=\left\{\begin{array}{ll}
\One & \text{if $i\in A$, $j\leq m+n-\iota^\vee_{i-\#\{c\in \Ac\mid c<i\}}(v)$ or $i\in \Ac$, $j-i\leq \iota_{i-\#\{a\in A\mid a<i\}}(w)$},\\
\reg-\One & \text{if $i\in A$, $j=m+n- \iota_{i-\#\{c\in \Ac\mid c<i\}}^\vee(v)+1$, }\\
\reg-\One & \text{if $i\in \Ac$, $j-i-1\in \Big\{ \iota_{i-\#\{a\in A\mid a<i\}}(w),\#\{c\in \Ac\mid c>i\}\Big\}$,}\\
\reg & \text{otherwise.}
\end{array}\right.$$
When applying $\pi_B$, we apply multiple copies of the Borel subgroup $B_{m+n}$ to the rows.  In a row $i\in A$, the start (from left to right) is a sequence of (possibly 0) $\One$'s and the end is a sequence of (possibly 0) $\reg$'s, and if there are both, then they must be separated by exactly one $\reg-\One$.  By Corollary \ref{SupercharacterFactorization} this is a standard row for a supercharacter, so by (\ref{SupercharacterClassFunctions}) the right action of $b\in B_{m+n}$ leaves such a row invariant.  

In a row $i\in \Ac$, we again begin the row with (possibly 0) $\One$'s and end the row with (possibly 0) $\reg$'s.  If both occur, then they must be separated by either one $\reg-\One$ or two $\reg-\One$'s which in turn are separated by (possibly 0) $\reg$'s. If there is at most one $\reg-\One$ in such a row, then $b\in B_{m+n}$ again leaves it invariant.  We may therefore focus on the rows of the form
$$[\hspace{1cm}\cdot\ \underbrace{\One\ \cdots \ \One}_{\text{$r$ terms}}\ \reg-\One \underbrace{\reg\ \cdots\ \reg}_{\text{$s$ terms}}\ \reg-\One\ \underbrace{\reg\ \cdots\ \reg}_{\text{$t$ terms}}].$$
where $i\in \Ac$, and $r,s,t$ could all be zero.  In this case, a computation as in (\ref{ProjectionComputation}) shows that 
\begin{align*}\sum_{b\in B_{m+n}}& [\hspace{1cm}\cdot\ \underbrace{\One\ \cdots \ \One}_{\text{$r$ terms}}\ \reg-\One \underbrace{\reg\ \cdots\ \reg}_{\text{$s$ terms}}\ \reg-\One\ \underbrace{\reg\ \cdots\ \reg}_{\text{$t$ terms}}]b\\
&=|B_{m+n}|\frac{q-1}{q}[\hspace{1cm}\cdot\ \underbrace{\One\ \cdots \ \One}_{\text{$r$ terms}}\ \reg-\One \underbrace{\reg\ \cdots\ \reg}_{\text{$s+t+1$ terms}}].
\end{align*}
Thus,
$$\pi_B(\psi_A^{w,v})=\left(\frac{q-1}{q}\right)^{\#\{i\in \Ac\mid \iota_{i-\#\{a\in A\mid a<i\}}(w)<\#\{c\in \Ac\mid c>i\}\}}\tilde\psi_A^{w,v},
$$
where for $1\leq i<j\leq m+n$,
$$(\tilde\psi_{A}^{w,v})_{ij}=\left\{\begin{array}{ll}
\One & \text{if $i\in A$, $j\leq m+n-\iota^\vee_{i-\#\{c\in \Ac\mid c<i\}}(v)$ or $i\in \Ac$, $j-i\leq \iota_{i-\#\{a\in A\mid a<i\}}(w)$},\\
\reg-\One & \text{if $i\in A$, $j=m+n- \iota_{i-\#\{c\in \Ac\mid c<i\}}^\vee(v)+1$, }\\
\reg-\One & \text{if $i\in \Ac$, $j-i=\iota_{i-\#\{a\in A\mid a<i\}}(w)+1$,}\\
\reg & \text{otherwise.}
\end{array}\right.$$
By inspection,
$$\tilde\psi_{A}^{w,v}=\chi^{w\bowtie_A v}\quad \text{and}\quad \left(\frac{q-1}{q}\right)^{\#\{i\in \Ac\mid \iota_{i-\#\{a\in A\mid a<i\}}(w)<\#\{c\in \Ac\mid c>i\}\}}\frac{1}{\chi^w(0)\chi^v(0)\chi^\faith_{\fkr_A}(0)}=\frac{1}{\chi^{w\bowtie_A v}(0)}.$$

\end{proof}

\begin{remark} \label{ProductHeuristic} On diagrams, we shuffle the inverse permutations or the columns of the diagrams according to the columns picked out by $A$ (indicated in gray in the example, below).  For example,
$$
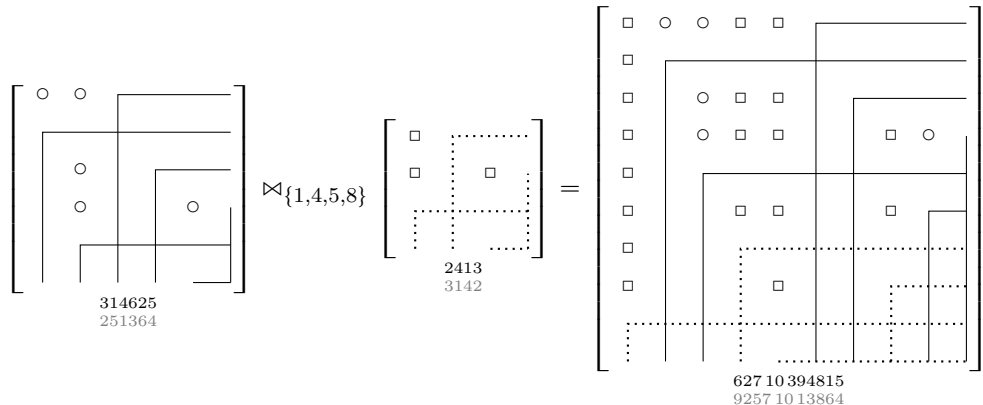
\underset{314625\atop \color{gray} 251364}{\left[\begin{tikzpicture}[scale=.5,baseline=1.65cm]
\foreach \i/\j in {1/3,2/1,3/4,4/6,5/2,6/5}
	\draw (\j,1) -- (\j,7-\i) -- (6,7-\i); 
\foreach \i/\j in {1/1,1/2,3/2,4/2,4/5}
	\node at (\j,7-\i) {$\circ$};
\end{tikzpicture}\right]}\bowtie_{\{1,4,5,8\}} 
\underset{2413\atop \color{gray} 3142}{\left[\begin{tikzpicture}[scale=.5,baseline=1.2cm]
\foreach \i/\j in {1/2,2/4,3/1,4/3}
	\draw[dotted,thick] (\j,1) -- (\j,5-\i) -- (4,5-\i); 
\foreach \i/\j in {1/1,2/1,2/3}
	\node at (\j,5-\i) {$\square$};
\end{tikzpicture}\right]}=
\underset{627\hspace{1pt} 10\hspace{1pt}  394815\atop \color{gray} 9257\hspace{1pt} 10\hspace{1pt}  13864}{\left[\begin{tikzpicture}[scale=.5,baseline=2.7cm]
\foreach \i/\j in {2/2,3/7,6/9,1/6,4/10,5/3}
	\draw (\j,1) -- (\j,11-\i) -- (10,11-\i); 
\foreach \i/\j in {7/4,9/1,10/5,8/8}
	\draw[dotted,thick] (\j,1) -- (\j,11-\i) -- (10,11-\i); 
\foreach \i/\j in {1/1,2/1,3/1,4/1,5/1,6/1,7/1,8/1,1/4,3/4,4/4,6/4,1/5,3/5,4/5,6/5,8/5,4/8,6/8}
	\node at (\j,11-\i) {$\square$};
\foreach \i/\j in {1/2,1/3,3/3,4/3,4/9}
	\node at (\j,11-\i) {$\circ$};
\end{tikzpicture}\right]}
$$
Thus, we can think of the $\bowtie_A$ as a shifted shuffle of the columns of the Rothe diagrams.  
\end{remark}

By the above remark, if we first invert the permutations, we actually start shuffling the permutations, so we use the $\star$-duality.

\begin{theorem} \label{GoingUpCombinatorics}
Let $A\subseteq \{1,2,\ldots, m+n\}$ with $|A|=n$ and complement $A'$.  Then  for $w\in S_m$ and $v\in S_n$,
$$\Sfl_{\fkut_A^\vee\times \fkut_A}^{\fkut_n}\Big((\chi^w)^\star\otimes (\chi^{v})^\star\Big)^\star=\chi^{w\shuffle_A v}.$$
\end{theorem}
\begin{proof}
By Lemma \ref{ProductLemma} (b),
$$\Sfl_{\fkut_A^\vee\times \fkut_A}^{\fkut_n}\Big((\chi^w)^\star\otimes (\chi^{v})^\star\Big)^\star=\chi^{(w^{-1}\bowtie_A v^{-1})^{-1}}.$$
Thus, it suffices to show that $w\shuffle_A v=(w^{-1}\bowtie_A v^{-1})^{-1}$.  By Lemma \ref{ProductLemma}(a),
\begin{align*}
(w^{-1}\bowtie_A v^{-1})^{-1}(j)&=\left\{\begin{array}{ll} w(\#\{i\in \Ac\mid i\leq j\}) & \text{if $j\in \Ac$,}\\
m+v(\#\{a\in A\mid a\leq j\}) & \text{if $j\in A$.} \end{array}\right.\\
&=(w\shuffle_A v) (j),
\end{align*}
as desired.
\end{proof}

\subsection{The functor $\Dela$ that goes down}

Let $A\subseteq \{1,2,\ldots, m+n\}$ with $|A|=n$.
The Frobenius adjoint to inflation is \emph{deflation}
$$\begin{array}{rccc}
 \Def_{\fkut_A^\vee\times \fkut_A}^{\fkp_A} : &  \scf(\fkp_A)& \longrightarrow & \scf(\fkut_A^\vee\times \fkut_{A}) \\
& \chi & \mapsto & \begin{array}{c@{\ }c@{\ }c} \fkut_A^\vee \times \fkut_A & \rightarrow & \CC\\ (u',u) & \mapsto & \dd\frac{1}{|\fkl_A|}\sum_{l\in \fkl_A} \chi(l+u'+u).\end{array} \end{array}$$
We will also use \emph{collapsing}
$$\begin{array}{rccc}
 \Col_{\fkp_A}^{\fkut_{m+n}} : &  \scf(\fkut_n)& \longrightarrow & \scf(\fkp_A) \\
& \chi & \mapsto & \begin{array}{c@{\ }c@{\ }c} \fkp_A & \rightarrow & \CC\\ y & \mapsto & \dd\frac{1}{|\fkr_A|}\sum_{r\in \fkr_A} \frac{\chi^\faith_{\fkr_A}(r)}{\chi^\faith_{\fkr_A}(0)}\chi(y+r),\end{array} \end{array}$$
which does not seem to be a standard functor in the literature.  

Their composition is the \emph{delapsing} functor
$$\Dela_{\fkut_A^\vee\times \fkut_A}^{\fkut_{m+n}} = \Def_{\fkut_A^\vee\times \fkut_A}^{\fkp_A} \circ \Col_{\fkp_A}^{\fkut_{m+n}}.$$
\begin{example}  Consider the supercharacter
$$\chi^{5317426}=\chi^{\iota^{-1}(2,4,1,2,0,1,0)}=
\left[\begin{smallmatrix}
\cdot & \One & \One & \reg-\One & \reg & \reg & \reg\\
 & \cdot & \One & \One & \One & \One &  \reg-\One \\
 & & \cdot & \One & \reg-\One & \reg& \reg \\
&  & & \cdot & \One & \One & \reg-\One \\
&  & & & \cdot & \reg-\One & \reg  \\
&  & & & & \cdot & \One \\
&  & & & & & \cdot  
\end{smallmatrix}\right].$$
In this case, it is instructive to consider different subsets of $\{1,2,\ldots,7\}$.  For $A=\{2,4,6\}$, 
\begin{align*} \Col_{\fkp_A}^{\fkut_7} (\chi^{5317426})&=\frac{1}{|\fkr_A|} \sum_{r\in \fkr_A}  \frac{\chi^\faith_{\fkr_A}(r)}{\chi^\faith_{\fkr_A}(0)} \chi^{5317426}(r)\left[\begin{smallmatrix}
\cdot & \One & \One & \reg-\One & \bullet & \bullet & \bullet \\
 &  \color{gray} 2 & \One & \One & \One & \One &  \reg-\One \\
 & & \cdot & \One & \reg-\One & \bullet & \bullet \\
&  & &\color{gray} 4 & \One & \One & \reg-\One \\
&  & & & \cdot & \reg-\One & \bullet  \\
&  & & & & \color{gray} 6& \One \\
&  & & & & & \cdot  
\end{smallmatrix}\right]\\
&=\left[\begin{smallmatrix}
\cdot & \One & \One & \reg-\One & \bullet & \bullet & \bullet \\
 &  \color{gray} 2 & \One & \One & \One & \One &  \reg-\One \\
 & & \cdot & \One & \reg-\One & \bullet & \bullet \\
&  & &\color{gray} 4 & \One & \One & \reg-\One \\
&  & & & \cdot & \reg-\One & \bullet  \\
&  & & & & \color{gray} 6& \One \\
&  & & & & & \cdot  
\end{smallmatrix}\right].
\end{align*}
since the sum is only nonzero when $r=0$ (the $\bullet$-entries are the ones in $R_A$).  Apply the deflation functor to get
$$\Dela_{\fkp_A}^{\fkut_7} (\chi^{5317426})=\frac{1}{|\fkl_A|} \sum_{l\in \fkl_A}\chi^{5317426}(l) \left[\begin{smallmatrix}
\cdot & \One & \One & \reg-\One & \bullet & \bullet & \bullet \\
 &  \color{gray} 2 & \star & \star & \star & \One &  \reg-\One \\
 & & \cdot & \One & \reg-\One & \bullet & \bullet \\
&  & &\color{gray} 4 & \star & \star & \reg-\One \\
&  & & & \cdot & \reg-\One & \bullet  \\
&  & & & & \color{gray} 6& \star \\
&  & & & & & \cdot  
\end{smallmatrix}\right]= \left[\begin{smallmatrix}
\cdot & \One & \One & \reg-\One & \bullet & \bullet & \bullet \\
 &  \color{gray} 2 & \star & \star & \star & \One &  \reg-\One \\
 & & \cdot & \One & \reg-\One & \bullet & \bullet \\
&  & &\color{gray} 4 & \star & \star & \reg-\One \\
&  & & & \cdot & \reg-\One & \bullet  \\
&  & & & & \color{gray} 6& \star \\
&  & & & & & \cdot  
\end{smallmatrix}\right],$$
since $\chi^{5317426}$ is trivial on all elements of $\fkl_A$ (the $\star$-entries are the ones in $L_A$). 

On the other hand, for $B=\{2,3,4\}$, 
\begin{align*} \Dela_{\fkp_B}^{\fkut_7} (\chi^{5317426})&=\Def_{\fkut_B^\vee\times \fkut_B}^{\fkp_B}\left(
\left[\begin{smallmatrix}
\cdot & \One & \One & \reg-\One & \bullet & \bullet & \bullet\\
 & \color{gray} 2 & \One & \One & \One & \One &  \reg-\One \\
 & & \color{gray} 3 & \One & \reg-\One & \reg& \reg \\
&  & &\color{gray} 4 & \One & \One & \reg-\One \\
&  & & & \cdot & \reg-\One & \reg  \\
&  & & & & \cdot & \One \\
&  & & & & & \cdot  
\end{smallmatrix}\right]\right)\\
&=\frac{1}{|\fkl_A|} \sum_{l\in \fkl_A}\chi^{5317426}(l) 
\left[\begin{smallmatrix}
\cdot & \One & \One & \reg-\One & \bullet & \bullet & \bullet\\
 & \color{gray} 2 & \star & \star & \star & \One &  \reg-\One \\
 & & \color{gray} 3 & \star & \star  & \star & \reg \\
&  & &\color{gray} 4 & \star & \star & \star \\
&  & & & \cdot & \reg-\One & \reg  \\
&  & & & & \cdot & \One \\
&  & & & & & \cdot  
\end{smallmatrix}\right]\\
&=0
\end{align*}
since
$$\frac{1}{|\fkl_A|} \sum_{l\in \fkl_A}\chi^{5317426}(l) =\frac{1}{|\fkl_A|}\Big(\sum_{l_{35}\in \FF_q} (\reg-\One)(l_{35})\Big) \sum_{l\in \fkl_A\atop l_{35}=0}\chi^{5317426}(l)=0,$$
by the orthogonality of characters of $\FF_q^+$. Lastly, for $C=\{1,4,5\}$,
$$\Col_{\fkp_C}^{\fkut_7} (\chi^{5317426})=\frac{1}{|\fkr_C|} \sum_{r\in \fkr_C}  \frac{\chi^\faith_{\fkr_C}(r)}{\chi^\faith_{\fkr_C}(0)} \chi^{5317426}(r)\left[\begin{smallmatrix}
 \color{gray} 1 & \One & \One & \reg-\One & \reg & \reg & \reg\\
 & \cdot & \One & \One & \One & \bullet &  \bullet \\
 & & \cdot & \One & \reg-\One & \bullet & \bullet  \\
&  & &  \color{gray} 4 & \One & \One & \reg-\One \\
&  & & &  \color{gray} 5 & \reg-\One & \reg  \\
&  & & & & \cdot & \One \\
&  & & & & & \cdot  
\end{smallmatrix}\right]=0,$$
since 
$$\frac{1}{|\fkr_C|} \sum_{r\in \fkr_C}  \frac{\chi^\faith_{\fkr_C}(r)}{\chi^\faith_{\fkr_C}(0)} \chi^{5317426}(r)=\frac{1}{|\fkr_C|} \Big(\sum_{r_{26}\in \FF_q} (\reg-\One)(r_{26})\One(r_{26}) \Big)\sum_{r\in \fkr_C\atop r_{26}=0}  \frac{\chi^\faith_{\fkr_C}(r)}{\chi^\faith_{\fkr_C}(0)} \chi^{5317426}(r)=0.$$
\end{example}

The following lemma formalizes exactly when we get nonzero values under the delapsing functor.

\begin{lemma}\label{SupportLemma}
Let $A\subseteq \{1,2,\ldots, m+n\}$ with $|A|=n$. For $w\in S_{m+n}$,  
\begin{enumerate}
\item[(a)] If $\Dela_{\fkut_A^\vee\times \fkut_{A}}^{\fkut_{m+n}}(\chi^w)\neq 0$, then
$$\{(a,a+\iota_a(w)+1)\mid a\in A\}\subseteq U_A\quad \text{and}\quad \{(j,j+\iota_j(w))\mid j\in \Ac, \iota_j(w)\neq 0\}\subseteq U_A^\vee.$$
\item[(b)] If $\Dela_{\fkut_A^\vee\times \fkut_{A}}^{\fkut_{m+n}}(\chi^w)\neq 0$, then for each $i$, $w(i)\in A$ implies $i>m$.
\end{enumerate}
\end{lemma}

\begin{proof} (a) We prove the contrapositive in two cases.  
Suppose there exists $a\in A$ with $(a,a+\iota_a(w)+1)\notin U_A$.  Then $(a,a+\iota_a(w)+1)=(a_0,b_0)\in L_A$, so $b_0-a_0=\iota_{a_0}(w)+1$.  By Corollary \ref{SupercharacterFactorization}, for $x\in \fkut_A^\vee\times \fkut_A\subseteq \fkut_{m+n}$, 
\begin{align*}
\Dela_{\fkut_A^\vee\times \fkut_{A}}^{\fkut_{m+n}}&(\chi^w)(x)\\
&=(\reg_{\FF_q^+}-\One_{\FF_q^+})\Big(\frac{1}{q}\sum_{t\in \FF_q}t\Big)\hspace{-.25cm} \prod_{1\leq i<j\leq m+n\atop j-i\leq \iota_i(v)}\hspace{-.25cm} \One_{\FF_q^+}(x_{ij}) \hspace{-.25cm}\prod_{1\leq i<j\leq m+n\atop{ j-i= \iota_i(v)+1\atop (i,j)\neq (a_0,b_0)}} \hspace{-.25cm}(\reg_{\FF_q^+}-\One_{\FF_q^+})(x_{ij}) \hspace{-.25cm}\prod_{1\leq i<j\leq m+n\atop j-i> \iota_i(v)+1}\hspace{-.25cm}\reg_{\FF_q^+}(x_{ij}),
\end{align*}
where
$$(\reg_{\FF_q^+}-\One_{\FF_q^+})\Big(\frac{1}{q}\sum_{t\in \FF_q}t\Big)=\frac{1}{q}\Big(q-1+(q-1)(-1)\Big)=0.$$

Suppose instead that there exists $j\in \Ac$ with $\iota_j(w)\neq 0$ and $(j,j+\iota_j(w))\notin U_A^\vee$.  Then $(j,j+\iota_j(w))=(i_0,j_0)\in R_A$ with $j_0-i_0<\iota_{i_0}(w)+1$; let $j_0$ be minimal with this property.   For $x\in \fkut_A^\vee\times \fkut_A\subseteq \fkut_{m+n}$, 
\begin{align*}
&\Dela_{\fkut_A^\vee\times \fkut_{A}}^{\fkut_n}(\chi^w)(x)\\
&=(\One_{\FF_q^+})\Big(\frac{1}{q}\sum_{t\in \FF_q}\frac{(\reg_{\FF_q^+}-\One_{\FF_q^+})(t)}{(\reg_{\FF_q^+}-\One_{\FF_q^+})(0)}t\Big)\hspace{-.25cm} \prod_{1\leq i<j\leq m+n\atop { j-i\leq \iota_i(v)\atop (i,j)\neq (i_0,j_0)}}\hspace{-.25cm} \One_{\FF_q^+}(x_{ij}) \hspace{-.25cm}\prod_{1\leq i<j\leq m+n\atop j-i= \iota_i(v)+1} \hspace{-.25cm}(\reg_{\FF_q^+}-\One_{\FF_q^+})(x_{ij}) \hspace{-.25cm}\prod_{1\leq i<j\leq m+n\atop j-i> \iota_i(v)+1}\hspace{-.25cm}\reg_{\FF_q^+}(x_{ij}),
\end{align*}
where
$$(\One_{\FF_q^+})\Big(\frac{1}{q}\sum_{t\in \FF_q}\frac{(\reg_{\FF_q^+}-\One_{\FF_q^+})(t)}{(\reg_{\FF_q^+}-\One_{\FF_q^+})(0)}t\Big)=\frac{1}{q}\Big(1+(q-1)\frac{-1}{q-1}\Big)=0.$$

(b) If $\Dela_{\fkut_A^\vee\times \fkut_{A}}^{\fkut_n}(\chi^w)\neq 0$, then both conditions from (a) must be satisfied.

Let $i_0$ be maximal such that $w(i_0)\notin A$.  Then $(w(i_0),w(i_0)+\iota_{w(i_0)}(w))\in U_A^\vee$ implies
$$\#\{a\in A\mid a>w(i_0)\}\leq m+n-\iota_{w(i_0)}(w)-w(i_0)=\#\{h>i_0\mid w(h)>w(i_0)\}.$$
In other words,
\begin{equation}\label{Local1}
\#\{w(h)\in A\mid h<i_0,w(h)>w(i_0)\}\leq \#\{h>i_0\mid w(h) \notin A,w(h)>w(i_0)\}=0.
\end{equation}
If there is no $w(h)\in A$ with $h<i_0$, then we are done.  Otherwise, let $h_0$ be minimal such that $w(h_0)\in A$.  Then by (\ref{Local1}) we must have $w(h_0)<w(i_0)$.  Since $(w(h_0),w(h_0)+\iota_{w(i_0)}(w)+1)\in U_A$, we have
$$\#\{a\in A\mid a>w(h_0)\}\geq m+ n-\iota_{w(h_0)}(w)-w(h_0)=\#\{i>h_0\mid w(i)>w(h_0)\},$$
or
$$0=\#\{w(i)\in A\mid i<h_0,w(i)>w(h_0)\}\geq \#\{w(i)\notin A\mid i>h_0,w(i)>w(h_0)\}\geq 1$$
a contradiction.  Thus, there is no $w(h)\in A$ with $h<i_0$.
\end{proof}

This lemma implies that the delapsing functor gives a standardized deconcatenation on the supercharacter basis.

\begin{theorem} \label{GoingDownCombinatorics}
Let $A\subseteq \{1,2,\ldots, m+n\}$ with $|A|=n$ and complement $A'$.   For $w\in S_{m+n}$,
$$\Dela_{\fkut_A^\vee\times \fkut_A}^{\fkut_{m+n}} (\chi^w)=\left\{\begin{array}{ll} \chi^{w_{\leq m}}\otimes \chi^{w_{>m}} & \text{if 
$w^{-1}(A)=\{m+1,\ldots, m+n\}$,}\\ 0 & \text{otherwise.}\end{array}\right.$$
\end{theorem}

\begin{proof}   By Lemma \ref{SupportLemma} (b), we may assume $w(i)\in A$ only if $i>m$. By Lemma \ref{SupportLemma} (a), if $a\in A$ and $(a,b)\in L_A$, then $b-a\leq\iota_a(w)$, and if $i\in \Ac$ and $(i,j)\in R_A$, then $j-i\geq \iota_i(w)+1$.   Thus, when we apply Corollary \ref{SupercharacterFactorization} to $\chi^w$, for $(a,b)\in L_A$ we have
$$\One_{\FF_q^+}\Big(\frac{1}{q}\sum_{t\in \FF_q} t\Big)=\frac{1}{q}q=1,$$
and for $(i,j)\in R_A$ either
$$(\reg_{\FF_q^+}-\One_{\FF_q^+})\Big(\frac{1}{q}\sum_{t\in \FF_q}\frac{(\reg_{\FF_q^+}-\One_{\FF_q^+})(t)}{(\reg_{\FF_q^+}-\One_{\FF_q^+})(0)}t\Big)=\frac{1}{q}\Big((q-1)+(q-1)\frac{(-1)^2}{q-1}\Big)=1,$$
or
$$\reg_{\FF_q^+}\Big(\frac{1}{q}\sum_{t\in \FF_q}\frac{(\reg_{\FF_q^+}-\One_{\FF_q^+})(t)}{(\reg_{\FF_q^+}-\One_{\FF_q^+})(0)}t\Big)=\frac{1}{q} q=1\quad \text{or}\quad \reg_{\FF_q^+}\Big(\frac{1}{q}\sum_{t\in \FF_q}\frac{\reg_{\FF_q^+}(t)}{\reg_{\FF_q^+}(0)}t\Big)=\frac{1}{q} \reg_{\FF_q}(0)=1.$$
By the definition of delapsing we therefore have
\begin{align*}
\Dela_{\fkut_A^\vee\times \fkut_A}^{\fkut_{m+n}} (\chi^w)(u',u)= \bigg(&\prod_{(i,j)\in U_A^\vee\atop j-i\leq \iota_i(w)}\hspace{-.25cm} \One_{\FF_q^+}(u'_{ij}) \hspace{-.25cm}\prod_{(i,j)\in U_A^\vee\atop j-i= \iota_i(w)+1} \hspace{-.25cm}(\reg_{\FF_q^+}-\One_{\FF_q^+})(u'_{ij}) \hspace{-.25cm}\prod_{(i,j)\in U_A^\vee\atop j-i> \iota_i(w)+1}\hspace{-.25cm}\reg_{\FF_q^+}(u'_{ij})\bigg)\\
&\cdot  \bigg(\prod_{(i,j)\in U_A\atop j-i\leq \iota_i(w)}\hspace{-.25cm} \One_{\FF_q^+}(u_{ij}) \hspace{-.25cm}\prod_{(i,j)\in U_A\atop j-i= \iota_i(w)+1} \hspace{-.25cm}(\reg_{\FF_q^+}-\One_{\FF_q^+})(u_{ij}) \hspace{-.25cm}\prod_{(i,j)\in U_A\atop j-i> \iota_i(w)+1}\hspace{-.25cm}\reg_{\FF_q^+}(u_{ij})\bigg).
\end{align*}

Using the isomorphism (\ref{CheckIsomorphism}),
$$\chi^{w_{\leq m}}(\tau'_A(u'))=\chi^w(u'),$$
since $(j-\#\{a\in A\mid a<i\})-(i-\#\{a\in A\mid a<i\})=j-i$ and $\iota_{j}(w)=\iota_{j-\#\{a<j\mid a\in A\}}(w_{\leq m})$ for all $j\in \Ac$.

Next using the isomorphism (\ref{UnCheckIsomorphism}),
$$\chi^{w_{>m}}(\tau_A(u))=\chi^w(u), $$
since $j-m-(i-\#\{b\in \Ac\mid b<i \})=j-i-\#\{b\in \Ac\mid b>i\}$ and
$$\iota_{j-\#\{b\in \Ac\mid b<j \}}(w_{>m})=\iota_{j}(w)-\#\{b\in \Ac\mid b>j \}.$$

We can conclude that
$$\Dela_{\fkut_A^\vee\times \fkut_A}^{\fkut_{m+n}} (\chi^w)(u',u)=\chi^{w_{\leq m}}(u')\chi^{w_{>m}}(u),$$
as desired.
\end{proof}

\begin{remark} \label{CoProductHeuristic}
For deconcatenating, we split off the last $k$ rows (if we are deconcatenating $k$ digits off the end),
$$\underset{831\hspace{1pt}10\hspace{1pt}746925\atop \color{gray} 3926\hspace{1pt}10\hspace{1pt}57184}{\left[\begin{tikzpicture}[scale=.5,baseline=2.7cm]
\foreach \i/\j in {7/6,8/9,9/2,10/5}
	\draw[dotted,thick]  (\j,1) -- (\j,11-\i) -- (10,11-\i); 
\foreach \i/\j in {1/8,2/3,3/1,6/4,4/10,5/7}
	\draw (\j,1) -- (\j,11-\i) -- (10,11-\i); 
\foreach \i/\j in {7/2,7/5,8/2,8/5}
	\node at (\j,11-\i) {$\square$};
\foreach \i/\j in {1/1,1/2,1/3,1/4,1/5,1/6,1/7,4/2,4/4,4/5,4/6,4/7,4/9,5/2,5/4,5/5,5/6, 2/1,2/2,6/2}
	\node at (\j,11-\i) {$\circ$};
\end{tikzpicture}\right]}=\underset{521643\atop \color{gray} 326514}{\left[\begin{tikzpicture}[scale=.5,baseline=1.65cm]
\foreach \i/\j in {1/5,2/2,3/1,4/6,5/4,6/3}
	\draw (\j,1) -- (\j,7-\i) -- (6,7-\i); 
\foreach \i/\j in {1/1,1/2,1/3,1/4,2/1,4/3,4/4,5/3}
	\node at (\j,7-\i) {$\circ$};
\end{tikzpicture}\right]}\ .\ 
\underset{3412\atop \color{gray} 3412}{\left[\begin{tikzpicture}[scale=.5,baseline=1.2cm]
\foreach \i/\j in {1/3,2/4,3/1,4/2}
	\draw[dotted,thick] (\j,1) -- (\j,5-\i) -- (4,5-\i); 
\foreach \i/\j in {1/1,1/2,2/1,2/2}
	\node at (\j,5-\i) {$\square$};
\end{tikzpicture}\right]}.$$
This can  be viewed as a deshuffle of the inverse permutations picked out by the columns of $A$ (as indicated by the gray inverses of the permutations).
\end{remark}

\subsection{The Malvenuto--Reutenauer algebra and $\scf(\fkut)$}

Using the functors from the previous section, we may define a representation theoretic product and coproduct on $\scf(\fkut)$.   For the product we define
$$\begin{array}{ccc} \scf(\fkut_m)\otimes \scf(\fkut_n) & \longrightarrow & \scf(\fkut_{m+n})\\
\alpha\otimes \beta & \mapsto & \dd\sum_{A\subseteq \{1,2,\ldots, m+n\}\atop |A|=n} \Sfl_{\fkut_A^\vee\times \fkut_A}^{\fkut_{m+n}} (\alpha^\star\otimes \beta^\star)^\star,\end{array}$$
and coproduct given by
$$\begin{array}{rccc}  \Delta: & \scf(\fkut_n) & \longrightarrow & \dd\bigoplus_{m=0}^n\scf(\fkut_{n-m})\otimes \scf(\fkut_m)\\
&\alpha & \mapsto & \dd\sum_{A\subseteq \{1,2,\ldots, n\}} \Dela_{\fkut_A^\vee\times \fkut_A}^{\fkut_{n}} (\alpha).\end{array}$$

From Theorems \ref{GoingUpCombinatorics} and \ref{GoingDownCombinatorics} we obtain the following corollary.

\begin{corollary}\label{SupercharacterStructureConstants}\hfill

\begin{enumerate}
\item[(a)] For $w\in S_m$ and $v\in S_n$,
$$\chi^w\chi^v=\sum_{y\in w\shuffle v} \chi^y.$$
\item[(b)] For $w\in S_n$,
$$\Delta(\chi^w)=\sum_{m=0}^n \chi^{w_{\leq m}}\otimes \chi^{w_{>m}}.$$
\end{enumerate}
\end{corollary}

In particular, by comparing with (\ref{FQSymProduct}) and (\ref{FQSymCoProduct}) we obtain a Hopf algebra isomorphism to the Malvenuto--Reutenauer algebra. 

\begin{corollary}\label{MainIsomorphism}
The linear function
$$\begin{array}{r@{\ }c@{\ }c@{\ }c} \mathrm{ch}: & \scf(\fkut) & \longrightarrow & \FQSym\\ & \chi^w & \mapsto & F_w\end{array}$$
is a Hopf algebra isomorphism.
\end{corollary}

\begin{remark}
By computing the structure constants we get away with not checking that $\Sfl$ with the $\star$-involution and $\Dela$ are Hopf compatible.  There should be an algebraic proof of this, but that would likely also require a better representation theoretic interpretation of the $\star$-involution.
\end{remark}

\subsection{The dual Hopf algebra $\scf(\fkut)^*$}
 
It is well-known that the Malvenuto--Reutenauer algebra $\FQSym$ is self-dual -- albeit in an interesting fashion (see, for example, \cite{GR}); however, we think it is instructive to study this duality from a representation theoretic point of view.  The dual Hopf algebra can be constructed from a choice of bilinear form on $\scf(\fkut)$.  However, the representation theoretic construction supplies a canonical inner product on characters given by
$$\begin{array}{r@{\ }c@{\ }c@{\ }c} 
\langle\cdot,\cdot\rangle : &\scf(\fkut_m)\otimes \scf(\fkut_n) & \longrightarrow & \CC\\
& \gamma\otimes \psi & \mapsto & \dd\left\{\begin{array}{@{}ll@{}}\dd\frac{1}{|\fkut_n|} \sum_{u\in \fkut_n} \gamma(u)\overline{\psi(u)} & \text{if $m=n$},\\
0 & \text{otherwise.}\end{array}\right.
\end{array}$$
Note that  for $w,v\in S_n$,
$$\langle \delta_w,\delta_v\rangle=\left\{\begin{array}{@{}ll@{}}\frac{|\Cl_w|}{|\fkut_n|} & \text{if $v=w$,} \\ 0 & \text{otherwise,}\end{array}\right. \qquad \text{and}\qquad \langle \chi^w,\chi^v\rangle=\left\{\begin{array}{@{}ll@{}} \chi^w(0) & \text{if $v=w$,} \\ 0 & \text{otherwise.}\end{array}\right. $$
  However, the permutation characters are not orthogonal (e.g. they all contain the trivial character).  To construct the dual Hopf algebra, it will be helpful to construct the adjoint functors with respect to this inner product.

\begin{proposition}
Let $A\subseteq \{1,2,\ldots, m+ n\}$ with $|A|=n$.  For $\psi\in \scf(\fkut_{m+n})$ and $\varphi\otimes \eta\in \scf(\fkut_A^\vee)\otimes \scf(\fkut_A)$,
$$\langle\Sfl_{\fkut_A^\vee\times \fkut_A}^{\fkut_{m+n}} (\varphi\otimes \eta),\psi\rangle=\langle \varphi\otimes \eta, \Dela_{\fkut_A^\vee\times \fkut_A}^{\fkut_{m+n}}(\psi)\rangle.$$ 
\end{proposition}
\begin{proof}
By definition,
\begin{align*}
\langle &\Sfl(\varphi\otimes\eta),\psi\rangle = \frac{1}{|\fkut_{m+n}|} \sum_{x\in \fkut_{m+n}} \Sfl_A(\varphi\otimes\eta)(x) \overline{\psi(x)}\\
&=\frac{1}{|\fkut_{m+n}|} \sum_{x\in \fkut_{m+n}}\frac{1}{|B_{m+n}|^{m+n}}\sum_{\underline{b}\in B_{m+n}^{m+n}}  \bigg(\Big(\Inf_{\fkut_{A^\vee}\times \fkut_{A}}^{\fkp_A}(\varphi\otimes\eta)\Big) \otimes (\chi^\faith_{\fkr_A})^*\bigg)\underline{b}(x) \overline{\psi(x)}.
\end{align*}
Note that for $y\in \cf(\fkut_{m+n})$, $x\in \fkut_{m+n}$ and $\underline{b}=(b_1,\ldots, b_n)\in B_{m+n}$, 
$$(y\underline{b})(x)=y((\underline{b}x^T)^T),$$
where $x^T$ is the transpose of $x$ and $\underline{b}$ acts on the columns of $z\in \fkut_n$ individually, so
$$(b_1,\ldots, b_{m+n})[z_{*1}\ z_{*2} \ \cdots\ z_{*m+n}]=[b_1z_{*1}\ b_2z_{*2} \ \cdots\ b_{m+n}z_{*m+n}].$$
Thus,
\begin{align*}
\langle \Sfl(\varphi\otimes\eta),\psi\rangle
 &= \frac{1}{|\fkut_{m+n}||B_{m+n}|^{m+n}}\hspace{-.45cm} \sum_{x\in \fkut_{m+n}\atop \underline{b}\in B_{m+n}^{m+n}}\hspace{-.15cm} \bigg(\Big(\Inf_{\fkut_{A^\vee}\times \fkut_{A}}^{\fkp_A}(\varphi\otimes\eta)\Big) \otimes (\chi^\faith_{\fkr_A})^*\bigg)((\underline{b}(x)^T)^T) \overline{\psi(x)}\\
 &= \frac{1}{|\fkut_{m+n}||B_{m+n}|^{m+n}}\hspace{-.45cm}  \sum_{x\in \fkut_{m+n}\atop \underline{b}\in B_{m+n}^{m+n}} \hspace{-.15cm}  \bigg(\Big(\Inf_{\fkut_{A^\vee}\times \fkut_{A}}^{\fkp_A}(\varphi\otimes\eta)\Big) \otimes (\chi^\faith_{\fkr_A})^*\bigg)(x) \overline{\psi((\underline{b}(x)^T)^T)}\\
  &= \frac{1}{|\fkut_{m+n}||B_{m+n}|^{m+n}}\hspace{-.45cm}  \sum_{x\in \fkut_{m+n}\atop \underline{b}\in B_{m+n}^{m+n}} \hspace{-.15cm} \bigg(\Big(\Inf_{\fkut_{A^\vee}\times \fkut_{A}}^{\fkp_A}(\varphi\otimes\eta)\Big) \otimes (\chi^\faith_{\fkr_A})^*\bigg)(x) \overline{(\psi \underline{b}) (x)}.
  \end{align*}
  However, $\psi\in \scf(\fkut_{m+n})$, so $\psi\underline{b}=\psi$ and 
\begin{align*}
\langle  \Sfl(\varphi\otimes\eta),\psi\rangle
&= \frac{1}{|\fkut_{m+n}|} \sum_{u\in \fkut_A^\vee, v\in \fkut_A\atop l\in \fkl_A,r\in \fkr_A} \bigg(\Big(\Inf_{\fkut_{A^\vee}\times \fkut_{A}}^{\fkp_A}(\varphi\otimes\eta)\Big) \otimes (\chi^\faith_{\fkr_A})^*\bigg)(l+u+v+r) \overline{\psi(l+u+v+r)}\\
&=\frac{1}{|\fkut_{m+n}|} \sum_{u\in \fkut_A^\vee, v\in \fkut_A\atop l\in \fkl_A,r\in \fkr_A} \varphi(u)\eta(v)\frac{\chi^\faith_{\fkr_A}(r)}{\chi^\faith_{\fkr_A}(0)} \overline{\psi(l+u+v+r)}\\
&=\langle \varphi\otimes \eta, \Dela_{\fkut_A^\vee\times \fkut_A}^{\fkut_{m+n}}(\psi)\rangle,
\end{align*}
as desired.
\end{proof}

We therefore obtain the dual Hopf algebra structure on $\scf(\fkut)$ with product
$$\begin{array}{ccc} \scf(\fkut_m)\otimes \scf(\fkut_n) & \longrightarrow & \scf(\fkut_{m+n})\\
\alpha\otimes \beta & \mapsto & \dd\sum_{A\subseteq \{1,2,\ldots, m+n\}\atop |A|=n}\Sfl_{\fkut_A^\vee\times \fkut_A}^{\fkut_{m+n}} (\alpha\otimes \beta),\end{array}$$
and coproduct 
$$\begin{array}{rccc}  \Delta: & \scf(\fkut_n) & \longrightarrow & \dd\bigoplus_{m=0}^n\scf(\fkut_{n-m})\otimes \scf(\fkut_m)\\
&\alpha & \mapsto & \dd\sum_{A\subseteq \{1,2,\ldots, n\}} \Dela_{\fkut_A^\vee\times \fkut_A}^{\fkut_{n}} (\alpha^\star)^\star.\end{array}$$
With respect to our inner product $\langle\cdot,\cdot \rangle$ we have dual bases
\begin{description}
\item[Superclass identifiers.] By the orthogonality relation,
$$\delta_w^*=\frac{1}{|\Cl_w|}\delta_w.$$
\item[Supercharacters.] By the orthogonality relation,
$$(\chi^w)^*=\frac{1}{\chi^w(0)}\chi^w.$$
\end{description}

\begin{remark}
From the results of the previous section, the product and coproduct of $\frac{\chi^w}{\chi^w(0)}$ shuffle and deconcatenate along the columns of the diagram of $w$.  Alternatively, we shuffle inverses and deconcatenate the inverse word.  It follows that 
$$\begin{array}{ccc} \scf(\fkut) & \longrightarrow & \scf(\fkut)\\ \chi^w & \mapsto & (\chi^w)^\star\end{array}$$
is a Hopf algebra isomorphism, where the standard Hopf structure on the left becomes the dual structure from this section on the right.  Note that this differs somewhat from the standard isomorphism since we are using a slightly different pairing.
\end{remark}

\section{The permutation character basis of $\scf(\fkut)$} \label{PermutationCharacterStructure}

Up to this point, we have worked almost exclusively with the supercharacter basis.  However, we have two other canonical bases at our disposal.  This section computes the structure constants for the permutation character basis (\ref{PermutationCharacterBasis}). In terms of our matrix representation of class functions (\ref{MatrixFunctions}), $\bar{\chi}^w$ is even  simpler than the supercharacters.  For example,
$$
\bar{\chi}^{5317426}=\bar{\chi}^{\iota^{-1}(2,4,1,2,0,1,0)}=
\left[\begin{smallmatrix}
\cdot & \One & \One & \reg & \reg & \reg & \reg\\
 & \cdot & \One & \One & \One & \One &  \reg \\
 & & \cdot & \One & \reg & \reg& \reg \\
&  & & \cdot & \One & \One & \reg \\
&  & & & \cdot & \reg & \reg  \\
&  & & & & \cdot & \One \\
&  & & & & & \cdot  
\end{smallmatrix}\right].
$$
Note the inversion table gives the lengths of the $\One$'s sequence in each row and the $\vee$-version of the inversion table (\ref{DualInversionTable}) gives the number of $\reg$'s in each row.  In other words,
$$|\fkut_w|=q^{|\iota(w)|}\qquad \text{and}\qquad \frac{|\fkut_n|}{|\fkut_w|}=q^{|\iota^\vee(w)|}.$$ 

\subsection{Coproduct}

The coproduct on this basis behaves as a weak standardized deconcatenation.    We begin with the analog of Lemma \ref{SupportLemma} in this context.  

\begin{lemma}\label{SubgroupSupportLemma}
Let $A\subseteq \{1,2,\ldots, m+ n\}$ with complement $\Ac$, $w\in S_{m+n}$.   If $\Dela_{\fkut_A^\vee\times \fkut_{A}}^{\fkut_{m+n}}(\bar\chi^w)\neq 0$, then $\{(j,j+\iota_j(w))\mid j\in \Ac, \iota_j(w)\neq 0\}\subseteq U_A^\vee.$
\end{lemma}

\begin{proof}
Suppose  that there exists $j\in \Ac$ with $\iota_j(w)\neq 0$ and $(j,j+\iota_j(w))\notin U_A^\vee$. 
 If $\fkut_v\geq \fkut_w$, then $(j,j+\iota_j(v))\in R_A$, and thus by part (a) of Lemma \ref{SupportLemma}, we have $\Dela_{\fkut_A^\vee\times \fkut_{A}}^{\fkut_{m+n}}(\chi^v)=0$. Therefore,
$$ \Dela_{\fkut_A^\vee\times \fkut_{A}}^{\fkut_{m+n}}(\bar\chi^w)= \Dela_{\fkut_A^\vee\times \fkut_{A}}^{\fkut_{m+n}}\left(\sum_{\fkut_w\subseteq \fkut_v}\chi^v\right)=\sum_{\fkut_w\subseteq \fkut_v}\Dela_{\fkut_A^\vee\times \fkut_{A}}^{\fkut_{m+n}}\left(\chi^v\right)=0,$$
as desired.
\end{proof}

Let $\mathrm{id}$ be the identity element of $S_n$, recall the notation $\iota^\vee(w)$ from (\ref{DualInversionTable}), and  for $A=\{a_1,\ldots,a_\ell\}\subseteq \{1,2,\ldots, n\}$, let 
$$\iota(w)_A=(\iota_{a_1}(w),\ldots,\iota_{a_\ell}(w)) \quad\text{where $a_1<a_2<\cdots<a_\ell$}.$$

\begin{theorem} \label{PermutationCharacterCoproduct} For $w\in S_{m+n}$,
$$\Delta(\bar{\chi}^w)=\sum_{A\sqcup \Ac= \{1,2,\ldots, m+n\}\atop\iota(w)_{\Ac}\text{ an inversion table}}
\bar{\chi}^{\iota^{-1}(\iota(w)_{\Ac})}\otimes \bar\chi^{(\iota^\vee)^{-1}(\iota^\vee(w)_A\wedge\iota^\vee(\mathrm{id}_A))},$$
where $\alpha\wedge \beta=\min(\alpha,\beta)$ denotes the meet of $\alpha$ and $\beta$ in (\ref{PermutationSequenceOrder}).
\end{theorem}
\begin{remark}
While the meet $\iota^\vee(w)_A\wedge\iota^\vee(\mathrm{id}_A)$ is in the lattice of all non-negative integer sequences, any such sequence $\alpha$ less than $\iota^\vee(\mathrm{id}_A)$ satisfies $\alpha=\iota^\vee(v)$ for some permutation $v\in S_{m+n}$.
\end{remark}

\begin{proof}  Let $(h,g)\in \Cl_{\beta}\times \Cl_{\alpha}$, where $\alpha$ and $\beta$ are inversion tables.
We have that 
$$\Dela_{\fkut^\vee_A\times \fkut_A}^{\fkut_{m+n}}(\bar{\chi}^w)(h,g)=\sum_{l\in \fkl_A,r\in\fkr_A}\frac{1}{|\fkl_A||\fkr_A|}\frac{\chi_{\fkr_A}^\faith(r)}{ \chi_{\fkr_A}^\faith(r)}\bar\chi^w(l+h+g+r).$$
However, by Lemma \ref{SubgroupSupportLemma}, $\fkr_A\cap \fkut_w=\{0\}$, so every term in the sum is zero for nontrivial $r\in \fkr_A$.  In addition, $\bar\chi^w(l+h+g)\neq 0$ if and only if  $l\in \fkut_w$, so
$$\beta_{\#\{i\in \Ac\mid i\leq  j\}}\leq \iota_j(w), \quad\text{for all $j\in \Ac$,}\quad \text{and} \quad \alpha^\vee_{\#\{i\in A\mid i\leq a\}}>\iota^\vee_a(w)\quad \text{for all $a\in A$.}$$

If $\alpha$ and $\beta$ satisfy these conditions, then
\begin{equation}\label{FirstVersion}
\Dela_{\fkut^\vee_A\times \fkut_A}^{\fkut_{m+n}}(\bar{\chi}^w)(h,g)=\frac{|\fkl_A\cap \fkut_w|}{|\fkl_A||\fkr_A|}\frac{|\fkut_n|}{|\fkut_w|}=\frac{|\fkl_A\cap \fkut_w|}{|\fkl_A||\fkr_A|}q^{|\iota^\vee(w)|}.
\end{equation}
On the other hand, $\bar{\chi}^{\iota(w)_{\Ac}}\otimes \bar\chi^{\iota^\vee(w)_A\wedge\iota^\vee(\mathrm{id}_A)}(h,g)\neq 0$ exactly when for all $j\in \Ac$, 
$$\beta_{\#\{i\in \Ac\mid i\leq  j\}}\leq \iota_j(w)
\quad \text{and for all $a\in A$,}\quad
\alpha^\vee_{\#\{i\in A \mid i\leq a\}}>\iota^\vee_a(w).$$
In this case,
\begin{equation}\label{QuotientProduct}
\bar{\chi}^{\iota(w)_{\Ac}}\otimes \bar\chi^{\iota^\vee(w)_A\wedge\iota^\vee(\mathrm{id}_A)}(h,g)=\frac{|\fkut_A^\vee|}{|\fkut_{\iota^{-1}(\iota(w)_{\Ac})}|} \frac{|\fkut_A|}{|\fkut_{(\iota^\vee)^{-1}(\iota^\vee(w)_A\wedge\iota^\vee(\mathrm{id}_A))}|}.
\end{equation}
Note that 
\begin{equation}\label{FirstQuotient}
\frac{|\fkut_A^\vee|}{|\fkut_{\iota^{-1}(\iota(w)_{\Ac})}|}=|\fkut_A^\vee|\prod_{j\in \Ac}\frac{1}{q^{\iota_j(w)}}=\frac{1}{|\fkr_A|}\prod_{j\in \Ac}\frac{q^{m+n-j}}{q^{\iota_j(w)}}=\frac{1}{|\fkr_A|}\prod_{j\in \Ac}q^{\iota_j^\vee(w)},
\end{equation}
where the second equality comes from the observation that row $j\in \Ac$ consists of $m+n-j$ entries from both $U_A^\vee$ and $R_A$.  On the other hand,
\begin{align}
\frac{|\fkut_A|}{|\fkut_{(\iota^\vee)^{-1}(\iota^\vee(w)_A\wedge\iota^\vee(\mathrm{id}_A))}|}&=\prod_{a\in A\atop \iota_a^\vee(w)\leq \#\{b\in A\mid b>a\}} \hspace{-.25cm}q^{\iota_a^\vee(w)}\hspace{-.25cm}\prod_{a\in A\atop \iota_a^\vee(w)>\#\{b\in A\mid b>a\}}\hspace{-.25cm}q^{\#\{ b\in A\mid b>a\}} \notag\\
&=\prod_{a\in A\atop \iota_a^\vee(w)\leq \#\{b\in A\mid b>a\}} \hspace{-.5cm}q^{\iota_a^\vee(w)}\hspace{-.5cm}\prod_{a\in A\atop \iota_a^\vee(w)>\#\{b\in A\mid b>a\}}\hspace{-.5cm}\frac{q^{\iota_a^\vee(w)}}{q^{\#\{(a,j)\in L_A\mid j-a>\iota_a(w)\}}}\notag\\
&=\frac{|\fkl_A\cap \fkut_w|}{|\fkl_A|}\prod_{a\in A} q^{\iota_a^\vee(w)},\label{SecondQuotient}
\end{align}
where for the last equality we pull out the denominator.  Plug (\ref{FirstQuotient}) and (\ref{SecondQuotient}) into (\ref{QuotientProduct}) to get equality with (\ref{FirstVersion}).
\end{proof}

\subsection{Product}

To compute the product on the permutation characters, we use the transition matrix to the supercharacter basis.  For this transition matrix we make use of the poset on permutations obtained by identifying permutation with their inversion tables (\ref{PermutationSequenceOrder}), and we denote the M\"obius function in this poset by $\mu$.    For $v\in S_m$ and $w\in S_n$,
\begin{align*}
\bar\chi^v\cdot \bar\chi^w & = \Big(\sum_{x\geq v} \chi^y\Big)\Big(\sum_{x'\geq w} \chi^{x'}\Big)\\
&=\sum_{x\geq v\atop x'\geq w} \sum_{y\in x\shuffle x'} \chi^y\\
&=\sum_{x\geq v\atop x'\geq w} \sum_{y\in x\shuffle x'} \sum_{z\geq y} \mu(y,z)\bar\chi^z\\
&=\sum_{z\in S_{m+n}} \bigg(\sum_{{y\leq z\atop y_{\{1,2,\ldots,m\}}\geq v}\atop  y_{\{m+1,\ldots,m+n\}}\geq w} \mu(y,z) \bigg)\bar\chi^z.
\end{align*}
where for $w\in S_{m+n}$, $A\subseteq \{1,2,\ldots, m+n\}$, $w_A\in S_{|A|}$ is the ``deshuffled" permutation given by
$$w_A^{-1}\Big(\#\{i\in A\mid i\leq j\}\Big)=w^{-1}(j)-\#\{i<w^{-1}(j)\mid w(i)\notin A\}\quad\text{for $j\in A$}.$$ 
For example,
$$(193458672)_{\{1,2,5,8\}}=\hspace{-.5cm}
\begin{tikzpicture}[baseline=-.1cm]
\foreach \x/\y in {1/1,2/9,3/3,4/4,5/5,6/8,7/6,8/7,9/2}
	\node[gray] (\x) at (\x/4,1) {$\y$};
\foreach \x/\y in {4/1,5/3,6/4,7/2}
	\node (\x0) at (\x/4,0) {$\y$};
\foreach \s/\t in {1/40,5/50,6/60,9/70}
	\draw[->] (\s) -- (\t);
\end{tikzpicture}\ .
$$

The goal of this section is to show 
$$\sum_{{y\leq z\atop y_{\{1,2,\ldots,m\}}\geq v}\atop  y_{\{m+1,\ldots,m+n\}}\geq w} \mu(y,z)\in \{-1,0,1\}.$$

By Proposition \ref{CoveringInversions}, a \emph{covering inversion} $(w(i),w(j))$ in $w$ is a pair with $i$  maximal such that $i< j$ and $w(i)>w(j)$.   Alternatively,  if $w'$ is the permutation obtained by switching $w(i)$ and $w(j)$, then
$$\iota_k(w')=\left\{\begin{array}{ll} 
\iota_k(w)  - 1 & \text{if $k=w(j)$,}\\
\iota_k(w) & \text{otherwise.}
\end{array}\right.$$
Thus, every $1\leq w(j)\leq n$ can be the second coordinate in at most one covering inversion (determined by whether $\iota_{w(j)}(w)>0$ or not).   We will refer to $w'$ as the permutation obtained by \emph{removing the covering inversion $(w(i),w(j))$ from $w$}.  Let
$$\coverings(w)=\{(w(i),w(j))\text{ a covering inversion of $w$}\}.$$
The following lemma is not used explicitly below, but underlies much of the intuition in what follows.  Specifically, it says that given any set of covering inversions, one may remove them in a specific order so that with each removal  the remaining covering inversions do not change.

\begin{lemma} \label{CoveringLemma} Let $w\in S_n$ and $I\subseteq \coverings(w)$.  Let $(w(i),w(j))\in I$ be selected with $w(j)$ minimal.  If $w'$ is the permutation with $(w(i),w(j))$ removed, then $I-\{(w(i),w(j))\}\subseteq \coverings(w')$.
\end{lemma}
\begin{proof}
Let $(w(k),w(l))\in I-\{(w(i),w(j))\}$ where by minimality of $w(j)$, we have $w(j)<w(l)$.  Then either $l<i$ or $l>j$.  In either case, if we remove $(w(i),w(j))$ from $w$, then the resulting permutation $w'$ will still have $(w(k),w(l))$ as a covering inversion. 
\end{proof}
By iteratively applying Lemma \ref{CoveringLemma} it makes sense to define for $C\subseteq \coverings(w)$, the permutation
$$w^{\remove(C)}=\text{ the permutation obtained by iteratively removing the covering inversions in $C$ from $w$,} $$
so that $\coverings\Big(w^{\remove{C}}\Big)=\coverings(w)-C$.
\begin{remark}
Note that since
$$\{z^{\remove(C)}\mid C\subseteq \coverings(z)\}$$
 is a Boolean lattice, we have that for $z\in S_n$ and $y\leq z$,
$$\mu(y,z)=\left\{\begin{array}{ll} (-1)^{|C|} & \text{if $y=z^{\remove(C)}$ for some $C\subseteq \coverings(z)$},\\ 0 & \text{otherwise.}\end{array}\right.$$
\end{remark}

If $v\in S_m$ and $w\in S_n$, let 
\begin{align*}
\deshuffle_{v,w}&=\{y\in S_{m+n}\mid y_{\{1,2,\ldots,m\}}\geq v,y_{\{m+1,\ldots, m+n\}}\geq w\}\\
\coveringsets^z_{v,w}&=\{C\subseteq \coverings(z)\mid z^{\remove(C)}\in \deshuffle_{v,w}\}.
\end{align*}
Our goal is to compute
\begin{equation} \label{ProductStrategy} \Coeff(\bar\chi^v\bar\chi^w;\bar\chi^z)=\sum_{C\in \coveringsets^z_{v,w}} (-1)^{|C|}.
\end{equation}
We will do this by iteratively partitioning $\coveringsets^z_{v,w}$ into blocks that contain an internal sign reversing bijection until we are left with at most 1 element.

\subsubsection{Free covering inversions}

A covering inversion $c\in \coverings(z)$ is called \emph{free} with respect to the pair $(v,w)$ if for all $C\subseteq \coverings(z)-\{c\}$,
$$C\in \coveringsets^z_{v,w}\qquad \text{if and only if} \qquad C\cup\{c\}\in \coveringsets^z_{v,w}.$$
To help get an intuition for free covering inversions we analyze the different types of covering inversions that occur.  For example, let
$$z=971458326\qquad \text{with}\qquad \coverings(z)=\{(9,7),(9,8),(7,1),(7,4),(7,5),(8,3),(8,6),(3,2)\},$$   
and consider the deshuffle
\begin{align*}
z_{\{1,2,3,4,5\}}&=14532\quad\text{with}\quad \iota(z_{\{1,2,3,4,5\}})=(0,3,2,0,0),\\
z_{\{6,7,8,9\}}&=4231\quad\text{with}\quad \iota(z_{\{6,7,8,9\}})=(3,1,1,0).
\end{align*}
Then $(3,2)$, $(9,7)$, and $(7,5)$ have fundamentally different behaviors in terms of $z_{\{1,2,3,4,5\}}$ and $z_{\{6,7,8,9\}}$, as outlined below.

\begin{description}
\item[Case 1.] If a covering inversion $(z(i),z(j))$ satisfies $z(i)\leq m$, then there is a corresponding covering inversion $(z(i),z(j))\in \coverings(z_{\{1,2,\ldots, m\}})$ and 
\begin{equation}\label{SmallIndicesCover}
\Big(z^{\remove(\{(z(i),z(j))\})}\Big)_{\{1,2\ldots, m\}}=(z_{\{1,2,\ldots,m\}})^{\remove(\{(z(i),z(j))\})}.
\end{equation}
In terms of inversion tables,
$$\iota_{z(j')}\Big((z^{\remove(\{(z(i),z(j))\})})_{\{1,2,\ldots,m\}}\Big)=\left\{\begin{array}{ll}
\iota_{z(j)}\big(z_{\{1,2,\ldots,m\}}\big)-1 & \text{if $j'=j$},\\
\iota_{z(j)}\big(z_{\{1,2,\ldots,m\}}\big) & \text{otherwise.}
\end{array}\right.$$
In this case, the removal of $(z(i),z(j))$ has no effect on $z_{\{m+1,\ldots, m+n\}}$.
\item[Case 2.] If a covering inversion $(z(i),z(j))$ satisfies $z(j)>m$, then there is a corresponding covering inversion  $(z(i)-m,z(j)-m)\in \coverings(z_{\{m+1,\ldots, m+n\}})$, and
\begin{equation}\label{LargeIndicesCover}
\Big(z^{\remove(\{(z(i),z(j))\})}\Big)_{\{m+1,\ldots, m+n\}}=(z_{\{m+1,\ldots,m+n\}})^{\remove(\{(z(i)-m,z(j)-m)\})}.
\end{equation}
In terms of inversion tables,
$$\iota_{z(j')}\Big((z^{\remove(\{(z(i),z(j))\})})_{\{m+1,\ldots,m+n\}}\Big)=\left\{\begin{array}{ll}
\iota_{z(j)}\big(z_{\{m+1,\ldots,m+n\}}\big)-1 & \text{if $j'=j$},\\
\iota_{z(j)}\big(z_{\{m+1,\ldots,m+n\}}\big) & \text{otherwise.}
\end{array}\right.$$
In this case, the removal of $(z(i),z(j))$ has not effect on $z_{\{1,2,\ldots, m\}}$.
\item[Case 3.]  If a covering inversion $(z(i),z(k))$ satisfies $z(i)>m$ and $z(k)\leq m$, then stranger things can happen.  In this case,
$$\iota_{z(j)}\Big((z^{\remove(\{(z(i),z(k))\})})_{\{1,2,\ldots,m\}}\Big)=\left\{\begin{array}{ll}
\iota_{z(j)}\big(z_{\{1,2,\ldots,m\}}\big)+1 & \text{if $i<j<k$},\\
\iota_{z(j)}\big(z_{\{1,2,\ldots,m\}}\big) & \text{otherwise.}
\end{array}\right.$$
So in this case, removing the covering inversion can increase the size of the deshuffled permutation, and the removal of $(z(i),z(j))$ has no effect on $z_{\{m+1,\ldots, m+n\}}$.
\end{description}

\begin{lemma}\label{FreeLemma}
Let $v\in S_m$, $w\in S_n$, and $z\in S_{m+n}$.  Then
\begin{enumerate}
\item[(a)] If $z$ has a free covering inversion with respect to $(v,w)$ then
$$\sum_{C\in \coveringsets^z_{v,w}} (-1)^{|C|}=0.$$
\item[(b)] Suppose $\coveringsets_{v,w}^z\neq \emptyset$.  If  $z$ has no free covering inversions with respect to $(v,w)$, then $z_{\{1,2\ldots,m\}}\ngtr v$ and $z_{\{m+1,\ldots,m+n\}}= w$.
\end{enumerate}
\end{lemma}
\begin{proof}
(a) If $\coveringsets^z_{v,w}=\emptyset$, then the sum is trivially 0.  Else, $C\in \coveringsets^z_{v,w}$ and we have a bijection between
$$\phi:\{C\in  \coveringsets^z_{v,w}\mid c\in C\} \longrightarrow\{D\in  \coveringsets^z_{v,w}\mid c\notin D\},$$
 such that $(-1)^{|\phi(C)|}=-(-1)^{|C|}$.  It follows that the sum is zero.
 
 (b) By (\ref{SmallIndicesCover}) and (\ref{LargeIndicesCover}) any covering inversion that is in those two cases respects the ordering on the deshuffled permutation.   Thus, if $z_{\{1,2\ldots,m\}}> v$ or $z_{\{m+1,\ldots,m+n\}}> w$, then there exists $(z(i),z(j))\in \coverings(z)$ with either $z(i)\leq m$ or $z(j)>m$ such that $\iota_{z(j)}(z)>\iota_{z(j)}(v.w)$, where
 $$v.w(k)=\left\{\begin{array}{ll} v(k) & \text{if $1\leq k\leq m$,}\\ w(k-m)+m & \text{if $m<k\leq m+n$}\end{array}\right.$$
is the shifted concatenation of $v$ with $w$; it follows that $(z(i),z(j))$ is free.  If $z_{\{m+1,\ldots,m+n\}}\ngeq w$, then there is no set of covering inversions $C$ such that $z^{\remove(C)}_{\{m+1,\ldots, m+n\}}\geq w$.   Thus, $z_{\{m+1,\ldots,m+n\}}=w$.
\end{proof}
\begin{corollary}
Suppose $\coveringsets^z_{v,w}$ has no free covering inversions and $C\in\coveringsets^z_{v,w}$.  Then $(z(i),z(j))\in C$ implies $z(j)\leq m$.
\end{corollary}

For a set of covering inversions $C\subseteq \coverings(z)$, we can represent each covering inversion on a row of nodes by denoting $(z(i),z(j))$ with an arc connecting the $i$th node to the $j$th node. For example,
$$z=971458326\qquad \text{with}\qquad \coverings(z)=\{(9,7),(9,8),(7,1),(7,4),(7,5),(8,3),(8,6),(3,2)\},$$
would be given by
$$
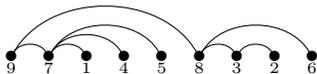
\begin{tikzpicture}[scale=.5]
\foreach \x/\y in {1/9,2/7,3/1,4/4,5/5,6/8,7/3,8/2,9/6}
	{\node (\y) at (\x,0) [inner sep=-1pt] {$\bullet$};
	\node at (\x,-.3) {$\scscs\y$};}
\foreach \s/\t in {9/7,9/8,7/1,7/4,7/5,8/3,8/6,3/2}
	\draw (\s) to  [in=120, out=60] (\t);
\end{tikzpicture}\ .$$
Note that by the definition of a covering inversion we cannot have $(w(i),w(k)),(w(j),w(l))\in C$ with $i<j<k<l$, since we would simultaneously have $w(k)>w(j)$ and $w(k)<w(j)$.  Thus, the arc diagram will always be crossing free.    A \emph{connected component} of $C\subseteq \coverings(z)$ is a nonempty set of arcs that form connected component of the graph.  
 
 \begin{lemma} \label{ComponentRemoval}
Let $C\subseteq \coverings(z)$ be a connected component with minimal element at position $i$.  
\begin{enumerate}
\item[(a)] If $z(i)\leq m$, then
$$\iota_{z(j)}(z^{\remove(C)}_{\{1,2,\ldots,m\}})=\left\{\begin{array}{ll}
\iota_{z(j)}\big(z_{\{1,2,\ldots,m\}}\big)-1 & \text{if $(z(i'),z(j))\in C$ for some $i'\geq i$},\\
\iota_{z(j)}\big(z_{\{1,2,\ldots,m\}}\big) & \text{otherwise.}
\end{array}\right.$$
\item[(b)] Suppose $z(i)>m$ and for all $j>i$ such that $(z(i),z(j))\in C$ suppose $z(j)\leq m$.  Let $k$ be the maximal position in $C$. Then
$$\iota_{z(j)}(z^{\remove(C)}_{\{1,2,\ldots,m\}})=\left\{\begin{array}{ll@{}}
\iota_{z(j)}\big(z_{\{1,2,\ldots,m\}}\big) & \text{if $j<i$, $j>k$ or $(z(i'),z(j))\in C$ for some $i'\geq i$},\\
\iota_{z(j)}\big(z_{\{1,2,\ldots,m\}}\big)+1 & \text{otherwise.}
\end{array}\right.$$
\end{enumerate}
\end{lemma}
\begin{proof}
These follow from the fact that if $i=j_0<j_1<\cdots<j_\ell$ are the points of $C$, then $z^{\remove(C)}$ is the permutation given by
$$z^{\remove(C)}(k)=\left\{\begin{array}{ll} z(j_{r+1}) & \text{if $k=j_r$, $r<\ell$,}\\
z(i) & \text{if $k=j_\ell$,}\\
z(k) & \text{otherwise.}\end{array}\right.$$
In the first case, $z(i)$ does not get removed when restricting to the subset, and in the second case it does.
\end{proof}

 \subsubsection{Nested covering inversions}
 
 We say a covering inversion $(z(j),z(k))$ \emph{nests} in $C\subseteq \coverings(z)$ if there exists $i\leq j<k<l$ such that $(z(i),z(l))\in C$.   Fix $v\in S_m$, $w\in S_n$ and $z\in S_{m+n}$ and assume $\coveringsets_{v,w}^z\neq \emptyset$ with no free covering inversions.  Let $C\in\coveringsets_{v,w}^z$.  We say a covering inversion $b\notin C$ is \emph{addable} if $b$ nests in $C$ and $C\cup \{b\}\in \coveringsets_{v,w}^z$.  Similarly, a covering inversion $b\in C$ is  \emph{removable} if $b$ nests in $C$ and $C-\{b\}\in \coveringsets_{v,w}^z$.  If there exists $C\in \coveringsets_{v,w}^z$ such that $b$ is addable (resp. removable) in $C$, then we say $b$ is addable (resp. removable) in $\coveringsets_{v,w}^z$.
 
Let 
\begin{equation}\label{CoreSet}
\coreset_{v,w}^z=\{C\in \coveringsets_{v,w}^z\mid C\text{ has no addable and no removable inversions}\},
\end{equation}
and
$$\core_{v,w}(z)=\left\{\begin{array}{ll} \dd\min_{C\in \coreset_{v,w}^z} \{|C|\} & \text{if $\coveringsets_{v,w}^z$ has no free covering inversions and $|\coreset_{v,w}^z|\notin 2\ZZ$,}\\
0 & \text{otherwise}.\end{array}\right.$$
The main result of this section is the following product formula.
\begin{theorem}\label{PermutationCharacterProduct}
For $v\in S_m$ and $w\in S_n$,
$$\bar\chi^v\cdot \bar\chi^w=\sum_{z\in S_{m+n} \atop \core_{v,w}(z)\neq 0} (-1)^{\core_{v,w}(z)} \bar\chi^z.$$
\end{theorem}

By Lemma \ref{FreeLemma} (a), we may assume for the remainder of this section that $\coveringsets_{v,w}^z$ has no free covering inversions.  The basic strategy of the proof is to construct a number of sign reversing bijections on subsets of $\coveringsets_{v,w}^z$ that slowly reduce the sum in (\ref{ProductStrategy}) into at most one term.  The first lemma sets up the underlying philosophy.
   
 \begin{lemma}  \label{BaseNesting} Assume $\coverings(z)$ has no free covering inversions.  Let $b\in \coverings(z)$ be removable in $\coveringsets_{v,w}^z$. Then the function
$$\begin{array}{ccc} \{C\in \coveringsets_{v,w}^{z}\mid b\text{ is removable in $C$}\} & \longrightarrow & \{D\in  \coveringsets_{v,w}^{z}\mid b\text{ is addable in $D$}\}\\
C & \mapsto & C-\{b\}
\end{array}$$
is a bijection.
\end{lemma}
\begin{proof}
The goal is to show the function is well-defined, or $C-\{b\}\in  \coveringsets_{v,w}^{z}$.  Since there are no free covering inversions, either $z(j)$ is in the same connected component as $z(i)$ or $z(j)\leq m$.  By Lemma \ref{ComponentRemoval}, any nested arc removed increases the inversion table over $\{1,2,\ldots, m\}$, so if $C\in  \coveringsets_{v,w}^{z}$, then so is $C-\{b\}$.
\end{proof}
 
 Thus, there is a subset $\cR(\{b\})\subseteq \coveringsets_{v,w}^z$ such that
$$\sum_{C\in \coveringsets_{v,w}^z} (-1)^{|C|}=\sum_{C\in \cR(\{b\})} (-1)^{|C|}.$$
In this case, 
$$\cR(\{b\})=\{C\in \coveringsets_{v,w}^z\mid \text{$b$ is neither addable nor removable in $C$}\}.$$
The following iterates this procedure.   

 Let $\cB$ be a set of  covering inversions removable in $\coveringsets_{v,w}^z$.  Let 
$$
\cR(\cB)=\left\{C\in \coveringsets_{v,w}^z\mid \text{$C$ has no addable or removable inversions in $\cB$}\right\}.
$$
A \emph{removal sequence} of a set of removable inversions $\cB$ is a bijection 
$$\begin{array}{rccc} \eta: & \{1,2,\ldots, |\cB|\} & \longrightarrow &\cB\\
& j & \mapsto & \eta(j)\end{array}$$
 such that for each $1\leq j\leq |\cB|$,
$$\eta(j)\text{  is removable in  $C$ for some $C\in\cR(\{b\in \cB\mid \eta^{-1}(b)<j\})$.}$$
In this situation, let
\begin{equation*}
\cK_j(\cB;\eta)=\cR(\{b\in \cB\mid \eta^{-1}(b)< j\})-\cR(\{b\in \cB\mid \eta^{-1}(b)\leq j\}).
\end{equation*}
We obtain a set partition
$$\{\cK_j(\cB;\eta)\mid 1\leq j\leq |\cB|\}\cup \cR(\cB)$$
of  $\coveringsets_{v,w}^z$.  
\begin{example}
Suppose $z=917426358$, $v=7142635$, and $w=21$.  Then
\begin{align*}
\coverings(z)&=
\begin{tikzpicture}[scale=.5,baseline=0cm]
\foreach \x/\l in {1/9,2/1,3/7,4/4,5/2,6/6,7/3,8/5,9/8}
	{\node (\l) at (\x,0) [inner sep=-1pt] {$\bullet$};
	\node at (\x,-.3) {$\scscs\l$};}
\foreach \s/\t in {9/1,9/8,9/7,7/4,7/6,4/2,6/3,6/5}
	\draw  (\s) to  [in=120, out=60] (\t);
\end{tikzpicture}\\
\coveringsets_{v,w}^z&=\left\{\begin{array}{@{}c@{}}
%ROW 1
\begin{tikzpicture}[scale=.3,baseline=0cm]
\foreach \x/\l in {1/9,2/1,3/7,4/4,5/2,6/6,7/3,8/5,9/8}
	{\node (\l) at (\x,0) [inner sep=-1pt] {$\scs\bullet$};
	\node at (\x,-.5) {$\scscs\l$};}
\foreach \s/\t in {9/7}
	\draw  (\s) to  [in=120, out=60] (\t);
\end{tikzpicture},
\begin{tikzpicture}[scale=.3,baseline=0cm]
\foreach \x/\l in {1/9,2/1,3/7,4/4,5/2,6/6,7/3,8/5,9/8}
	{\node (\l) at (\x,0) [inner sep=-1pt] {$\scs\bullet$};
	\node at (\x,-.5) {$\scscs\l$};}
\foreach \s/\t in {9/7,7/4}
	\draw  (\s) to  [in=120, out=60] (\t);
\end{tikzpicture},
\begin{tikzpicture}[scale=.3,baseline=0cm]
\foreach \x/\l in {1/9,2/1,3/7,4/4,5/2,6/6,7/3,8/5,9/8}
	{\node (\l) at (\x,0) [inner sep=-1pt] {$\scs\bullet$};
	\node at (\x,-.5) {$\scscs\l$};}
\foreach \s/\t in {9/7,7/4,4/2}
	\draw  (\s) to  [in=120, out=60] (\t);
\end{tikzpicture},
\begin{tikzpicture}[scale=.3,baseline=0cm]
\foreach \x/\l in {1/9,2/1,3/7,4/4,5/2,6/6,7/3,8/5,9/8}
	{\node (\l) at (\x,0) [inner sep=-1pt] {$\scs\bullet$};
	\node at (\x,-.5) {$\scscs\l$};}
\foreach \s/\t in {9/7,7/6}
	\draw  (\s) to  [in=120, out=60] (\t);
\end{tikzpicture}\\
%ROW 2
\begin{tikzpicture}[scale=.3,baseline=0cm]
\foreach \x/\l in {1/9,2/1,3/7,4/4,5/2,6/6,7/3,8/5,9/8}
	{\node (\l) at (\x,0) [inner sep=-1pt] {$\scs\bullet$};
	\node at (\x,-.5) {$\scscs\l$};}
\foreach \s/\t in {9/7,7/4,7/6}
	\draw  (\s) to  [in=120, out=60] (\t);
\end{tikzpicture},
\begin{tikzpicture}[scale=.3,baseline=0cm]
\foreach \x/\l in {1/9,2/1,3/7,4/4,5/2,6/6,7/3,8/5,9/8}
	{\node (\l) at (\x,0) [inner sep=-1pt] {$\scs\bullet$};
	\node at (\x,-.5) {$\scscs\l$};}
\foreach \s/\t in {9/7,7/4,7/6,4/2}
	\draw  (\s) to  [in=120, out=60] (\t);
\end{tikzpicture},
\begin{tikzpicture}[scale=.3,baseline=0cm]
\foreach \x/\l in {1/9,2/1,3/7,4/4,5/2,6/6,7/3,8/5,9/8}
	{\node (\l) at (\x,0) [inner sep=-1pt] {$\scs\bullet$};
	\node at (\x,-.5) {$\scscs\l$};}
\foreach \s/\t in {9/7,7/6,4/2}
	\draw  (\s) to  [in=120, out=60] (\t);
\end{tikzpicture},
\begin{tikzpicture}[scale=.3,baseline=0cm]
\foreach \x/\l in {1/9,2/1,3/7,4/4,5/2,6/6,7/3,8/5,9/8}
	{\node (\l) at (\x,0) [inner sep=-1pt] {$\scs\bullet$};
	\node at (\x,-.5) {$\scscs\l$};}
\foreach \s/\t in {9/7,7/6,6/3}
	\draw  (\s) to  [in=120, out=60] (\t);
\end{tikzpicture}\\
%ROW 3
\begin{tikzpicture}[scale=.3,baseline=0cm]
\foreach \x/\l in {1/9,2/1,3/7,4/4,5/2,6/6,7/3,8/5,9/8}
	{\node (\l) at (\x,0) [inner sep=-1pt] {$\scs\bullet$};
	\node at (\x,-.5) {$\scscs\l$};}
\foreach \s/\t in {9/7,7/4,7/6,6/3}
	\draw  (\s) to  [in=120, out=60] (\t);
\end{tikzpicture},
\begin{tikzpicture}[scale=.3,baseline=0cm]
\foreach \x/\l in {1/9,2/1,3/7,4/4,5/2,6/6,7/3,8/5,9/8}
	{\node (\l) at (\x,0) [inner sep=-1pt] {$\scs\bullet$};
	\node at (\x,-.5) {$\scscs\l$};}
\foreach \s/\t in {9/7,7/4,7/6,4/2,6/3}
	\draw  (\s) to  [in=120, out=60] (\t);
\end{tikzpicture},
\begin{tikzpicture}[scale=.3,baseline=0cm]
\foreach \x/\l in {1/9,2/1,3/7,4/4,5/2,6/6,7/3,8/5,9/8}
	{\node (\l) at (\x,0) [inner sep=-1pt] {$\scs\bullet$};
	\node at (\x,-.5) {$\scscs\l$};}
\foreach \s/\t in {9/7,7/6,4/2,6/3}
	\draw  (\s) to  [in=120, out=60] (\t);
\end{tikzpicture},
\begin{tikzpicture}[scale=.3,baseline=0cm]
\foreach \x/\l in {1/9,2/1,3/7,4/4,5/2,6/6,7/3,8/5,9/8}
	{\node (\l) at (\x,0) [inner sep=-1pt] {$\scs\bullet$};
	\node at (\x,-.5) {$\scscs\l$};}
\foreach \s/\t in {9/7,7/6,6/5}
	\draw  (\s) to  [in=120, out=60] (\t);
\end{tikzpicture}\\
%ROW 4
\begin{tikzpicture}[scale=.3,baseline=0cm]
\foreach \x/\l in {1/9,2/1,3/7,4/4,5/2,6/6,7/3,8/5,9/8}
	{\node (\l) at (\x,0) [inner sep=-1pt] {$\scs\bullet$};
	\node at (\x,-.5) {$\scscs\l$};}
\foreach \s/\t in {9/7,7/4,7/6,6/5}
	\draw  (\s) to  [in=120, out=60] (\t);
\end{tikzpicture},
\begin{tikzpicture}[scale=.3,baseline=0cm]
\foreach \x/\l in {1/9,2/1,3/7,4/4,5/2,6/6,7/3,8/5,9/8}
	{\node (\l) at (\x,0) [inner sep=-1pt] {$\scs\bullet$};
	\node at (\x,-.5) {$\scscs\l$};}
\foreach \s/\t in {9/7,7/4,7/6,4/2,6/5}
	\draw  (\s) to  [in=120, out=60] (\t);
\end{tikzpicture},
\begin{tikzpicture}[scale=.3,baseline=0cm]
\foreach \x/\l in {1/9,2/1,3/7,4/4,5/2,6/6,7/3,8/5,9/8}
	{\node (\l) at (\x,0) [inner sep=-1pt] {$\scs\bullet$};
	\node at (\x,-.5) {$\scscs\l$};}
\foreach \s/\t in {9/7,7/6,4/2,6/5}
	\draw  (\s) to  [in=120, out=60] (\t);
\end{tikzpicture},
\begin{tikzpicture}[scale=.3,baseline=0cm]
\foreach \x/\l in {1/9,2/1,3/7,4/4,5/2,6/6,7/3,8/5,9/8}
	{\node (\l) at (\x,0) [inner sep=-1pt] {$\scs\bullet$};
	\node at (\x,-.5) {$\scscs\l$};}
\foreach \s/\t in {9/7,7/6,6/3,6/5}
	\draw  (\s) to  [in=120, out=60] (\t);
\end{tikzpicture}\\
%ROW 5
\begin{tikzpicture}[scale=.3,baseline=0cm]
\foreach \x/\l in {1/9,2/1,3/7,4/4,5/2,6/6,7/3,8/5,9/8}
	{\node (\l) at (\x,0) [inner sep=-1pt] {$\scs\bullet$};
	\node at (\x,-.5) {$\scscs\l$};}
\foreach \s/\t in {9/7,7/4,7/6,6/3,6/5}
	\draw  (\s) to  [in=120, out=60] (\t);
\end{tikzpicture},
\begin{tikzpicture}[scale=.3,baseline=0cm]
\foreach \x/\l in {1/9,2/1,3/7,4/4,5/2,6/6,7/3,8/5,9/8}
	{\node (\l) at (\x,0) [inner sep=-1pt] {$\scs\bullet$};
	\node at (\x,-.5) {$\scscs\l$};}
\foreach \s/\t in {9/7,7/4,7/6,4/2,6/3,6/5}
	\draw  (\s) to  [in=120, out=60] (\t);
\end{tikzpicture},
\begin{tikzpicture}[scale=.3,baseline=0cm]
\foreach \x/\l in {1/9,2/1,3/7,4/4,5/2,6/6,7/3,8/5,9/8}
	{\node (\l) at (\x,0) [inner sep=-1pt] {$\scs\bullet$};
	\node at (\x,-.5) {$\scscs\l$};}
\foreach \s/\t in {9/7,7/6,4/2,6/3,6/5}
	\draw  (\s) to  [in=120, out=60] (\t);
\end{tikzpicture}
\end{array}\right\}
\end{align*}
Let 
$$\cB=\{(6,3),(4,2)\}\quad \text{with}\quad (\eta(1),\eta(2))=((6,3),(4,2)).$$
Then
\begin{align*}
\cK_1(\cB;\eta)&=\left\{\begin{array}{@{}c@{}}
%ROW 4
\begin{tikzpicture}[scale=.3,baseline=0cm]
\foreach \x/\l in {1/9,2/1,3/7,4/4,5/2,6/6,7/3,8/5,9/8}
	{\node (\l) at (\x,0) [inner sep=-1pt] {$\scs\bullet$};
	\node at (\x,-.5) {$\scscs\l$};}
\foreach \s/\t in {9/7,7/6,6/5}
	\draw  (\s) to  [in=120, out=60] (\t);
\end{tikzpicture},
\begin{tikzpicture}[scale=.3,baseline=0cm]
\foreach \x/\l in {1/9,2/1,3/7,4/4,5/2,6/6,7/3,8/5,9/8}
	{\node (\l) at (\x,0) [inner sep=-1pt] {$\scs\bullet$};
	\node at (\x,-.5) {$\scscs\l$};}
\foreach \s/\t in {9/7,7/4,7/6,6/5}
	\draw  (\s) to  [in=120, out=60] (\t);
\end{tikzpicture},
\begin{tikzpicture}[scale=.3,baseline=0cm]
\foreach \x/\l in {1/9,2/1,3/7,4/4,5/2,6/6,7/3,8/5,9/8}
	{\node (\l) at (\x,0) [inner sep=-1pt] {$\scs\bullet$};
	\node at (\x,-.5) {$\scscs\l$};}
\foreach \s/\t in {9/7,7/4,7/6,4/2,6/5}
	\draw  (\s) to  [in=120, out=60] (\t);
\end{tikzpicture},
\begin{tikzpicture}[scale=.3,baseline=0cm]
\foreach \x/\l in {1/9,2/1,3/7,4/4,5/2,6/6,7/3,8/5,9/8}
	{\node (\l) at (\x,0) [inner sep=-1pt] {$\scs\bullet$};
	\node at (\x,-.5) {$\scscs\l$};}
\foreach \s/\t in {9/7,7/6,4/2,6/5}
	\draw  (\s) to  [in=120, out=60] (\t);
\end{tikzpicture}\\
%ROW 5
\begin{tikzpicture}[scale=.3,baseline=0cm]
\foreach \x/\l in {1/9,2/1,3/7,4/4,5/2,6/6,7/3,8/5,9/8}
	{\node (\l) at (\x,0) [inner sep=-1pt] {$\scs\bullet$};
	\node at (\x,-.5) {$\scscs\l$};}
\foreach \s/\t in {9/7,7/6,6/3,6/5}
	\draw  (\s) to  [in=120, out=60] (\t);
\draw[thick] (6) to  [in=120, out=60] (3);
\end{tikzpicture},
\begin{tikzpicture}[scale=.3,baseline=0cm]
\foreach \x/\l in {1/9,2/1,3/7,4/4,5/2,6/6,7/3,8/5,9/8}
	{\node (\l) at (\x,0) [inner sep=-1pt] {$\scs\bullet$};
	\node at (\x,-.5) {$\scscs\l$};}
\foreach \s/\t in {9/7,7/4,7/6,6/3,6/5}
	\draw  (\s) to  [in=120, out=60] (\t);
\draw[thick] (6) to  [in=120, out=60] (3);
\end{tikzpicture},
\begin{tikzpicture}[scale=.3,baseline=0cm]
\foreach \x/\l in {1/9,2/1,3/7,4/4,5/2,6/6,7/3,8/5,9/8}
	{\node (\l) at (\x,0) [inner sep=-1pt] {$\scs\bullet$};
	\node at (\x,-.5) {$\scscs\l$};}
\foreach \s/\t in {9/7,7/4,7/6,4/2,6/3,6/5}
	\draw  (\s) to  [in=120, out=60] (\t);
\draw[thick] (6) to  [in=120, out=60] (3);
\end{tikzpicture},
\begin{tikzpicture}[scale=.3,baseline=0cm]
\foreach \x/\l in {1/9,2/1,3/7,4/4,5/2,6/6,7/3,8/5,9/8}
	{\node (\l) at (\x,0) [inner sep=-1pt] {$\scs\bullet$};
	\node at (\x,-.5) {$\scscs\l$};}
\foreach \s/\t in {9/7,7/6,4/2,6/3,6/5}
	\draw  (\s) to  [in=120, out=60] (\t);
\draw[thick] (6) to  [in=120, out=60] (3);
\end{tikzpicture}
\end{array}\right\}\\
\cK_2(\cB;\eta)&=\left\{\begin{array}{@{}c@{}}
\begin{tikzpicture}[scale=.3,baseline=0cm]
\foreach \x/\l in {1/9,2/1,3/7,4/4,5/2,6/6,7/3,8/5,9/8}
	{\node (\l) at (\x,0) [inner sep=-1pt] {$\scs\bullet$};
	\node at (\x,-.5) {$\scscs\l$};}
\foreach \s/\t in {9/7,7/6}
	\draw  (\s) to  [in=120, out=60] (\t);
\end{tikzpicture},
\begin{tikzpicture}[scale=.3,baseline=0cm]
\foreach \x/\l in {1/9,2/1,3/7,4/4,5/2,6/6,7/3,8/5,9/8}
	{\node (\l) at (\x,0) [inner sep=-1pt] {$\scs\bullet$};
	\node at (\x,-.5) {$\scscs\l$};}
\foreach \s/\t in {9/7,7/4,7/6}
	\draw  (\s) to  [in=120, out=60] (\t);
\end{tikzpicture},
\begin{tikzpicture}[scale=.3,baseline=0cm]
\foreach \x/\l in {1/9,2/1,3/7,4/4,5/2,6/6,7/3,8/5,9/8}
	{\node (\l) at (\x,0) [inner sep=-1pt] {$\scs\bullet$};
	\node at (\x,-.5) {$\scscs\l$};}
\foreach \s/\t in {9/7,7/6,6/3}
	\draw  (\s) to  [in=120, out=60] (\t);
\end{tikzpicture},
\begin{tikzpicture}[scale=.3,baseline=0cm]
\foreach \x/\l in {1/9,2/1,3/7,4/4,5/2,6/6,7/3,8/5,9/8}
	{\node (\l) at (\x,0) [inner sep=-1pt] {$\scs\bullet$};
	\node at (\x,-.5) {$\scscs\l$};}
\foreach \s/\t in {9/7,7/4,7/6,6/3}
	\draw  (\s) to  [in=120, out=60] (\t);
\end{tikzpicture}\\
\begin{tikzpicture}[scale=.3,baseline=0cm]
\foreach \x/\l in {1/9,2/1,3/7,4/4,5/2,6/6,7/3,8/5,9/8}
	{\node (\l) at (\x,0) [inner sep=-1pt] {$\scs\bullet$};
	\node at (\x,-.5) {$\scscs\l$};}
\foreach \s/\t in {9/7,7/6,4/2}
	\draw  (\s) to  [in=120, out=60] (\t);
\draw[thick] (4) to  [in=120, out=60] (2);
\end{tikzpicture},
\begin{tikzpicture}[scale=.3,baseline=0cm]
\foreach \x/\l in {1/9,2/1,3/7,4/4,5/2,6/6,7/3,8/5,9/8}
	{\node (\l) at (\x,0) [inner sep=-1pt] {$\scs\bullet$};
	\node at (\x,-.5) {$\scscs\l$};}
\foreach \s/\t in {9/7,7/4,7/6,4/2}
	\draw  (\s) to  [in=120, out=60] (\t);
\draw[thick] (4) to  [in=120, out=60] (2);
\end{tikzpicture},
\begin{tikzpicture}[scale=.3,baseline=0cm]
\foreach \x/\l in {1/9,2/1,3/7,4/4,5/2,6/6,7/3,8/5,9/8}
	{\node (\l) at (\x,0) [inner sep=-1pt] {$\scs\bullet$};
	\node at (\x,-.5) {$\scscs\l$};}
\foreach \s/\t in {9/7,7/6,4/2,6/3}
	\draw  (\s) to  [in=120, out=60] (\t);
\draw[thick] (4) to  [in=120, out=60] (2);
\end{tikzpicture},
\begin{tikzpicture}[scale=.3,baseline=0cm]
\foreach \x/\l in {1/9,2/1,3/7,4/4,5/2,6/6,7/3,8/5,9/8}
	{\node (\l) at (\x,0) [inner sep=-1pt] {$\scs\bullet$};
	\node at (\x,-.5) {$\scscs\l$};}
\foreach \s/\t in {9/7,7/4,7/6,4/2,6/3}
	\draw  (\s) to  [in=120, out=60] (\t);
\draw[thick] (4) to  [in=120, out=60] (2);
\end{tikzpicture}
\end{array}\right\}\\
\cR(\cB)&=\Big\{
\begin{tikzpicture}[scale=.3,baseline=0cm]
\foreach \x/\l in {1/9,2/1,3/7,4/4,5/2,6/6,7/3,8/5,9/8}
	{\node (\l) at (\x,0) [inner sep=-1pt] {$\scs\bullet$};
	\node at (\x,-.5) {$\scscs\l$};}
\foreach \s/\t in {9/7}
	\draw  (\s) to  [in=120, out=60] (\t);
\end{tikzpicture},
\begin{tikzpicture}[scale=.3,baseline=0cm]
\foreach \x/\l in {1/9,2/1,3/7,4/4,5/2,6/6,7/3,8/5,9/8}
	{\node (\l) at (\x,0) [inner sep=-1pt] {$\scs\bullet$};
	\node at (\x,-.5) {$\scscs\l$};}
\foreach \s/\t in {9/7,7/4}
	\draw  (\s) to  [in=120, out=60] (\t);
\end{tikzpicture},
\begin{tikzpicture}[scale=.3,baseline=0cm]
\foreach \x/\l in {1/9,2/1,3/7,4/4,5/2,6/6,7/3,8/5,9/8}
	{\node (\l) at (\x,0) [inner sep=-1pt] {$\scs\bullet$};
	\node at (\x,-.5) {$\scscs\l$};}
\foreach \s/\t in {9/7,7/4,4/2}
	\draw  (\s) to  [in=120, out=60] (\t);
\end{tikzpicture}
\Big\}.
\end{align*}
\end{example}

Note by the definition of $\cK_j(\cB;\eta)$, 
$$\cK_j(\cB;\eta)=\left\{C\in \cK_{j}(\cB;\eta)\mid  \eta(j)\text{ removable in $C$}\right\} \cup \{D\in  \cK_{j}(\cB;\eta)\mid \text{$\eta(j)$ addable in $D$}\}.$$
The following lemma finds sign-reversing bijection between these two non-intersecting subsets in the way suggested by Lemma \ref{BaseNesting}.

 \begin{lemma}
Suppose $\cB$ is a set of covering inversions  removable  in $\coveringsets_{v,w}^z$ that admits a removal sequence $\eta$.  Then for $1\leq j\leq |\cB|$, the function
$$\begin{array}{ccc} \left\{C\in \cK_{j}(\cB;\eta)\mid  \eta(j)\text{ removable in $C$}\right\} & \longrightarrow & \{D\in  \cK_{j}(\cB;\eta)\mid \text{$\eta(j)$ addable in $D$}\}\\
C & \mapsto & C-\{\eta(j)\}
\end{array}$$
is a bijection.
\end{lemma}
\begin{proof}
The goal is to show that the bijection is well-defined.  Suppose $C-\{\eta(j)\}\notin \cK_i(\cB;\eta)$ for some $i\leq j$.    Then $\eta(i)$ is either addable or removable in $C-\{\eta(j)\}$.  But adding $\eta(j)$ cannot affect whether $\eta(i)$ is addable or removable, so $C\in  \cK_i(\cB;\eta)$ and $i=j$.
\end{proof}
As a consequence of this iterated procedure,
\begin{equation} \label{LastIteration}
\sum_{C\in \coveringsets_{v,w}^z} (-1)^{|C|}=\sum_{C\in \cR(\cB)} (-1)^{|C|}.
\end{equation}
For $1\leq i\leq m+n$, let
$$\possibilities_{v,w}^z(i)=\left\{C\subseteq \coverings(z)\ \bigg|\ \begin{array}{@{}l@{}}\text{$C$ a connected component for some}\\ \text{$D\in  \coreset_{v,w}^z$ with smallest position $i$}\end{array}\right\},$$
where we recall the notation $\coreset_{v,w}^z$ from (\ref{CoreSet}).

\begin{lemma}\label{RemainingSets}
Let $\cB$ be a set of covering inversions removable in $\coveringsets_{v,w}^z$ maximal with the property that there exists a removal sequence $\eta$.  Then
\begin{enumerate}
\item[(a)] $\cR(\cB)=\coreset_{v,w}^z$,
\item[(b)] for $z(i)>m$, the set $\possibilities_{v,w}^z(i)$ forms a chain under containment, such that if $B$ covers $A$ in $\possibilities_{v,w}^z(i)$, then $|B-A|=1$.
\end{enumerate}
\end{lemma}
\begin{proof}
(a) By definition $\coreset_{v,w}^z\subseteq \cR(\cB)$.  Let $C\in \cR(\cB)$.  By the maximality of $\cB$, $C$ has no removable inversions.  Suppose $b$ is an addable inversion for $C$.  Then $D=C\cup\{b\}\in \cK_j(\cB;\eta)$ for some $1\leq j\leq |\cB|$.    By assumption $b\neq \eta(j)$.  Thus, $\eta(j)\in C$.  Since $\eta(j)$ is nested in $D$ it must also be so in $C$ (worst case it is nested in $b$, but $b$ is also nested).  Thus, $C\in \cK_j(\cB;\eta)$, a contradiction.

(b) Suppose $A, B\in \possibilities_{v,w}^z(i)$ with $B\neq \emptyset$.   Let $k$ be the largest position of $B$.    If $A=\emptyset$, then $A\subseteq B$.  Else,  let $A=\{(i,a_1),(a_1,a_2),\ldots,(a_{\ell-1},a_\ell)\}$ and $B=\{(i,b_1),(b_1,b_2),\ldots,(b_{r-1},b_r)\}$.  WLOG suppose $\ell\leq r$.  If $A\nsubseteq B$, then there exists a minimal $h<\ell$ such that $(a_h,a_{h+1})\neq (b_h,b_{h+1})$ (in particular, $a_h=b_h$ and $a_{h+1}\neq b_{h+1}$).  However, in this case either  $(a_h,a_{h+1})$ is addable in $B$ or $(b_h,b_{h+1})$ is addable in  $A$, both contradicting (a).  Thus, $A\subseteq B$.  If $B$ covers $A$, then by the above argument,  $(a_h,a_{h+1})= (b_h,b_{h+1})$ for all $h\leq \ell$.  Now $A,B\in  \possibilities_{v,w}^z(i)$ implies $A\cup \{(b_{\ell},b_{\ell+1})\}\in \possibilities_{v,w}^z(i)$, so $r=\ell+1$.
\end{proof}

\begin{proof}[Proof of Theorem \ref{PermutationCharacterProduct}] 
Choose a set $\cB$ of covering inversions removable in $\coveringsets_{v,w}^z$ maximal with the property that there exists a removal sequence $\eta$.     Note that by Lemma \ref{RemainingSets} (a) and  (\ref{LastIteration}),  we are interested in
$$\sum_{C\in \coreset_{v,w}^z} (-1)^{|C|}.$$
Let 
$$I=\{1\leq i\leq m+n\mid z(i)>m,z(i+1)\leq m\}.$$
Since $\coveringsets_{v,w}^z$ has no free covering inversions, every $C\in \coreset_{v,w}^z$ has connected components whose smallest positions are in $I$.  Write 
$$C=\bigcup_{i\in I} C_i,$$
 where $C_{i}\in \possibilities_{v,w}^z(i)$.      Conversely, each collection of choices $B_{i}\in  \possibilities_{v,w}^z(i)$ for $i\in I$ gives
 $$\bigcup_{i\in I} B_i\in \coreset_{v,w}^z.$$
 Let $\min\in \coreset_{v,w}^z$ be the choice where $\min_{i}\in \possibilities_{v,w}^z(i)$ is minimal for each $i\in I$.  
Then 
 $$\sum_{C\in \coreset_{v,w}^z} (-1)^{|C|}=\prod_{r=1}^\ell \bigg(\sum_{C_{i}\in \possibilities_{v,w}^z(i)} (-1)^{|C_{i}|}\bigg)=\left\{\begin{array}{ll} 
\dd \prod_{i\in I} (-1)^{|\min_{i}|} & \text{if $|\possibilities_{v,w}^z(i)|\notin2\ZZ$ for all $i\in I$,}\\ 0 & \text{otherwise.}
 \end{array}\right.$$
Finally, $\core_{v,w}(z)=\sum_{i\in I} |\min_{i}|$ if $|\possibilities_{v,w}^z(i)|\notin 2\ZZ$ for all $i\in I$ and 0 otherwise, recovering the result.
\end{proof}

\section{A final Hopf monoid remark}

The representation theoretic point of view suggests a Hopf monoid version of $\FQSym$, specifically a Hopf monoid in the monoidal category of complex vector species with a Cauchy product. We refer the reader to \cite[Chapter 8]{AM10} for more details on such Hopf monoids. Define a vector species
$$\begin{array}{rccc} \scf(\fkut):& \{\text{sets}\} & \longrightarrow & \{\text{vector spaces}\}\\
& A & \mapsto & \bigoplus_{\phi\in \cL[A]} \scf(\fkut_\phi),
\end{array}$$
where $\cL[A]$ is the set of linear orders on $A$ and $\fkut_\phi$ is the group of upper-triangular matrices with rows and columns indexed by $A$ in the order given by $\phi$. 

For disjoint sets $A$ and $B$ with $\phi\in \cL[A]$, $\tau\in\cL[B]$ and $\gamma\in \cL[A\cup B]$, define a product by the linear extension of
\begin{equation}\label{MonoidProduct}
\begin{array}{rccc}m_{B,A} :& \scf(\fkut_\phi) \otimes \scf(\fkut_\tau) & \longrightarrow & \dd\bigoplus_{\gamma\in \phi\shuffle \tau} \scf(\fkut_{\gamma})\\
& \psi \otimes \eta & \mapsto &\dd \sum_{\gamma\in \phi\shuffle \tau}   \Sfl_{\fkut_\phi\times \fkut_\tau}^{\fkut_{\gamma}} (\psi^\star\otimes\eta^\star)^\star,
\end{array}
\end{equation}
and coproduct by the linear extension of
\begin{equation}\label{MonoidCoProduct}
\begin{array}{rccc} \Delta_{B,A}: & \scf(\fkut_\gamma)  & \longrightarrow & \scf(\fkut_{\gamma_B})\otimes \scf(\fkut_{\gamma_A})\\
&\psi  & \mapsto & \Dela_{\fkut_{\gamma_B}\times \fkut_{\gamma_A}}^{\fkut_{\gamma}} (\psi).
\end{array}
\end{equation} 
\begin{remark}
Unlike in \cite{ABT}, the underlying functions on the towers of groups do not come from a Hopf structure on linear orders.  There will therefore not be an (obvious) Hadamard product in this case.
\end{remark}

\begin{theorem}
The species $\scf(\fkut)$ is a connected Hopf monoid under the product (\ref{MonoidProduct}) and (\ref{MonoidCoProduct}).
\end{theorem}
\begin{proof}  Since $\scf(\fkut)$ is connected, the unit and counit are inverses of one-another \cite[Proposition 8.11]{AM10}. It suffices to check the various compatibility conditions summarized in \cite[8.1--8.3]{AM10}  Most are straightforward checks, and we focus on the most interesting one
$$\Delta_{T,B}\circ m_{L,R}=(m_{LT,RT}\otimes m_{LB,RB})\circ (\Delta_{LT,LB}\otimes \Delta_{RT,RB}),$$
where $L\sqcup R=T\sqcup B$.  Let 
$$LT=L\cap T,\quad RT=R\cap T,\quad LB=L\cap B\quad \text{and} \quad RB=R\cap B.$$
For $v\in S_\phi$ and $w\in S_\tau$,
\begin{align*}
\mathrm{LHS}(\chi^v\otimes \chi ^w) &=\sum_{\gamma\in \phi\shuffle\tau} \Dela_{\fkut_{\gamma_T}\times \fkut_{\gamma_B}}^{\fkut_{\gamma}}(\chi^{v\shuffle_{\gamma^{-1}(R)} w})\\
&=\sum_{\gamma\in \phi\shuffle \tau\atop T=(v\shuffle_{\gamma^{-1}(R)} w)\circ\gamma(\{1,2,\ldots,|T|\})} \chi^{(v\shuffle_{\gamma^{-1}(R)} w)_T}\otimes\chi^{(v\shuffle_{\gamma^{-1}(R)} w)_B}.
\end{align*}
On the other hand,
\begin{align*}
(\Delta_{LT,LB}\otimes \Delta_{RT,RB})&(\chi^v\otimes \chi ^w) \\
&= \left\{\begin{array}{ll} \chi^{v_{LT}}\otimes \chi^{v_{LB}}\otimes \chi^{w_{RT}}\otimes \chi^{w_{RB}} & \begin{array}{@{}l@{}}\text{if $LT=\{v\circ\phi(j)\mid 1\leq j\leq |LT|\}$}\\  RT=\{w\circ\tau(j)\mid 1\leq j\leq |RT|\},\end{array}\\
0 & \text{otherwise,}\end{array}\right.
\end{align*}
so 
\begin{align*}
\mathrm{RHS}&(\chi^v\otimes \chi ^w)\\ &= \left\{\begin{array}{ll}
\dd \sum_{\alpha\in \phi_{LT}\shuffle \tau_{RT}\atop \beta\in \phi_{LB}\shuffle \tau_{RB}}\chi^{v_{LT}\shuffle_{\alpha^{-1}(RT)} w_{RT}}\otimes \chi^{v_{LB}\shuffle_{\beta^{-1}(RB)} w_{RB}} & \begin{array}{@{}l@{}}\text{if $LT=\{v\circ\phi(j)\mid 1\leq j\leq |LT|\}$}\\  RT=\{w\circ\tau(j)\mid 1\leq j\leq |RT|\},\end{array}\\
0 & \text{otherwise.}\end{array}\right.
\end{align*}
Note that by inspection both the LHS and RHS are multiplicity free.   Suppose 
$$\Big((v\shuffle_{\gamma^{-1}(R)} w)_T,(v\shuffle_{\gamma^{-1}(R)} w)_B\Big)$$
satisfies $\gamma\in \phi\shuffle \tau, T=(v\shuffle_{\gamma^{-1}(R)} w)\circ\gamma(\{1,2,\ldots,|T|\})\}$.  Then for $1\leq j\leq |T|$,
\begin{itemize}
\item $\gamma(j)\in L$ if and only if $v(\phi(\tilde j))\in LT$, where $\tilde{j}=j-\#\{i<j\mid \gamma(i)\in R\}$,
\item $\gamma(j)\in R$ if and only if $w(\tau(j'))\in RT$, where $j'=j-\#\{i<j\mid \gamma(i)\in L\}$.
\end{itemize}
Thus,
$$\Big((v\shuffle_{\gamma^{-1}(R)} w)_T,(v\shuffle_{\gamma^{-1}(R)} w)_B\Big)=\Big(v_{LT}\shuffle_{\gamma_T^{-1}(RT)} w_{RT},v_{LB}\shuffle_{\gamma_B^{-1}(RB)} w_{RB}\Big),$$
where $LT=\{v\circ\phi(j)\mid 1\leq j\leq |LT|\}$ and $RT=\{w\circ\tau(j)\mid 1\leq j\leq |RT|\}$. 

Conversely, if 
$$\Big(v_{LT}\shuffle_{\alpha^{-1}(RT)} w_{RT},v_{LB}\shuffle_{\beta^{-1}(RB)} w_{RB}\Big),$$
satisfies $\alpha\in \phi_{LT}\shuffle \tau_{RT}$, $\beta\in \phi_{LB}\shuffle \tau_{RB}$, $LT=\{v\circ\phi(j)\mid 1\leq j\leq |LT|\}$ and $RT=\{w\circ\tau(j)\mid 1\leq j\leq |RT|\}$, then
$$\Big(v_{LT}\shuffle_{\alpha^{-1}(RT)} w_{RT},v_{LB}\shuffle_{\beta^{-1}(RB)} w_{RB}\Big)=\Big((v\shuffle_{(\alpha.\beta)^{-1}(R)} w)_T,(v\shuffle_{(\alpha.\beta)^{-1}(R)} w)_B\Big),$$
where $T=(v\shuffle_{(\alpha.\beta)^{-1}(R)} w)\circ(\alpha.\beta)(\{1,2,\ldots,|T|\})\}$.
\end{proof} 

\begin{example}
Let  $\phi=(i,t,s)$, $\tau=(\textbf{n},\textbf{o})$, so $L=\{i,t,s\}$ and $R=\{\textbf{n},\textbf{o}\}$.  Let 
$$T=\{i,t,\textbf{n}\}\qquad \text{and}\qquad B=\{s,\textbf{o}\},\quad\text{so}\quad T\cup B=L\cup R.$$
Then for $w=(s,i,t)$ and $v=(\mathbf{o},\mathbf{n})$,
$$\begin{array}{ccc}
\scf(\fkut_\phi)\otimes \scf(\fkut_\tau) & \longrightarrow &\dd \bigoplus_{\gamma\in \phi\shuffle \tau} \scf(\fkut_\gamma)\\
\chi^{(s,i,t)}\otimes \chi^{(\mathbf{o},\mathbf{n})} & \mapsto &\begin{array}{@{}c@{}} 
\chi^{\overset{{\color{gray}(i,t,s,\mathbf{n},\mathbf{o})}}{(s,i,t,\mathbf{o},\mathbf{n})}}+ \chi^{\overset{{\color{gray}(i,t,\mathbf{n},s,\mathbf{o})}}{(s,i,\mathbf{o},t,\mathbf{n})}}+\chi^{\overset{{\color{gray}(i,\mathbf{n},t,s,\mathbf{o})}}{(s,\mathbf{o},i,t,\mathbf{n})}}
+\chi^{\overset{{\color{gray}(\mathbf{n},i,t,s,\mathbf{o})}}{(\mathbf{o},s,i,t,\mathbf{n})}}+ \chi^{\overset{{\color{gray}(i,t,\mathbf{n},\mathbf{o},s)}}{(s,i,\mathbf{o},\mathbf{n},t)}}\\+\chi^{\overset{{\color{gray}(i,\mathbf{n},t,\mathbf{o},s)}}{(s,\mathbf{o},i,\mathbf{n},t)}}
+\chi^{\overset{{\color{gray}(\mathbf{n},i,t,\mathbf{o},s)}}{(\mathbf{o},s,i,\mathbf{n},t)}}+ \chi^{\overset{{\color{gray}(i,\mathbf{n},\mathbf{o},t,s)}}{(s,\mathbf{o},\mathbf{n},i,t)}}+\chi^{\overset{{\color{gray}(\mathbf{n},i,\mathbf{o},t,s)}}{(\mathbf{o},s,\mathbf{n},i,t)}}
+\chi^{\overset{{\color{gray}(\mathbf{n},\mathbf{o},i,t,s)}}{(\mathbf{o},\mathbf{n},s,i,t)}}
\end{array}
\\
 & \longrightarrow & \dd \bigoplus_{\gamma\in \phi\shuffle \tau} \scf(\fkut_{\gamma_T})\otimes \scf(\fkut_{\gamma_B})\\
 & \mapsto & \begin{array}{@{}c@{}} 
0+ \chi^{(s,i,\mathbf{o})}\otimes\chi^{(t,\mathbf{n})}+\chi^{(s,\mathbf{o},i)}\otimes\chi^{(t,\mathbf{n})}
+\chi^{(\mathbf{o},s,i)}\otimes\chi^{(t,\mathbf{n})}+ \chi^{(s,i,\mathbf{o})}\otimes\chi^{(\mathbf{n},t)}\\+\chi^{(s,\mathbf{o},i)}\otimes\chi^{(\mathbf{n},t)}
+ \chi^{(\mathbf{o},s,i)}\otimes\chi^{(\mathbf{n},t)}+0
+0+0,
 \end{array}
\end{array}$$
where we've added the shuffles $\gamma$ in gray.  On the other hand,
$$\begin{array}{ccc}
\scf(\fkut_\phi)\otimes \scf(\fkut_\tau) & \longrightarrow &\dd  \scf(\fkut_{\phi_{\{i,t\}}})\otimes  \scf(\fkut_{\phi_{\{s\}}})\otimes  \scf(\fkut_{\tau_{\{\mathbf{n}\}}})\otimes  \scf(\fkut_{\tau_{\{\mathbf{o}\}}})\\
\chi^{\overset{{\color{gray}(i,t,s)}}{(s,i,t)}}\otimes \chi^{\overset{{\color{gray}(\mathbf{n},\mathbf{o})}}{(\mathbf{o},\mathbf{n})}} & \mapsto & \chi^{(s,i)}\otimes \chi^{(t)}\otimes \chi^{(\mathbf{o})}\otimes \chi^{(\mathbf{n})}\\
& \longrightarrow & \dd\bigoplus_{\alpha\in \gamma_{LT}\shuffle \gamma_{RT}\atop \beta\in \gamma_{LB}\shuffle \gamma_{RB}} \scf(\fkut_\alpha)\otimes \scf(\fkut_\beta)\\
& \mapsto &  \begin{array}{@{}c@{}} 
 \chi^{(s,i,\mathbf{o})}\otimes\chi^{(t,\mathbf{n})}+\chi^{(s,\mathbf{o},i)}\otimes\chi^{(t,\mathbf{n})}
+\chi^{(\mathbf{o},s,i)}\otimes\chi^{(t,\mathbf{n})}\\+  \chi^{(s,i,\mathbf{o})}\otimes\chi^{(\mathbf{n},t)}+\chi^{(s,\mathbf{o},i)}\otimes\chi^{(\mathbf{n},t)}
+\chi^{(\mathbf{o},s,i)}\otimes\chi^{(\mathbf{n},t)}
 \end{array}
 \end{array}$$
 Note that  if $T=\{i,t,\mathbf{o}\}$, then we get 0 in both cases.
\end{example}


\begin{thebibliography}{9}
\bibitem{AIM} M. Aguiar, C. Andr\'e, C. Benedetti, N. Bergeron, Z. Chen, P. Diaconis, A. Hendrickson, S.
	Hsiao, M. Isaacs, A. Jedwab, K. Johnson, G. Karaali, A. Lauve, T. Le, S. Lewis, H. Li, K.
	Magaard, E. Marberg, J-C. Novelli, A. Pang, F. Saliola, L. Tevlin, J-Y. Thibon, N. Thiem, V.
	Venkateswaran, C. R. Vinroot, N. Yan and M. Zabrocki, Supercharacters, symmetric functions in
	noncommuting variables, and related Hopf algebras. Adv. Math. 229 (2012) 2310--2337.	
\bibitem{AM10} M. Aguiar and S. Mahajan, Monoidal functors, species and Hopf algebras, volume 29 of CRM
Monograph Series. American Mathematical Society, Providence, RI, 2010.
\bibitem{ABT} M. Aguiar, N. Bergeron, and N. Thiem,  Hopf monoids from class functions on unitriangular matrices, Alg. Number Theory 7 (2013) 1743--1779.
\bibitem{AS05} M. Aguiar, F. Sottile, Structure of the Malvenuto-Reutenauer Hopf algebra of permutations,
Advances in Mathematics 192 (2005) 225--275.
\bibitem{Al} F. Aliniaeifard, Normal supercharacter theories and their supercharacters, J. Algebra 469 (2017) 464--484.
\bibitem{AT} F. Aliniaeifard and N. Thiem, The structure of normal lattice supercharacter theories, Algebraic Combinatorics 3 (2020) 1059--1078. 
\bibitem{AT2} F. Aliniaeifard and N. Thiem,  Pattern groups and a poset based Hopf monoid, J.~Comb.~Theory Ser.~A 172 (2020), 31 pages.
%\bibitem{BS} C. Benedetti and B. Sagan, Antipodes and involutions, Journal of Combinatorial Theory, Series A 148 (2017) 275--315.
\bibitem{BLL} N. Bergeron, T. Lam, and H. Li,  Combinatorial Hopf algebras and Towers of Algebras -- Dimension, Quantization and Functorality,  Algebr. Represent. Theory, 15 (2012) 675--696.
\bibitem{DI} P. Diaconis and M. Isaacs, Supercharacters and superclasses for algebra groups, Trans. Amer. Math. Soc. 360 (2008) 2359--2392.
\bibitem{GKal}
I.M. Gelfand, D. Krob, A. Lascoux, B. Leclerc, V.S.
  Retakh, and J.-Y. Thibon, Noncommutative symmetric functions, Adv.
  Math. 112 (1995) no.~2, 218--348. 
  \bibitem{GR} D. Grinberg and V. Reiner, Hopf algebras in combinatorics, (2020, V7) arXiv:1409.8356.
  \bibitem{FGK} C.F. Fowler, S.R. Garcia, and G. Karaali, Ramanujan sums as supercharacters, Ramanujan J. 35 (2014) 205--241.
\bibitem{Ges}
I.M. Gessel, Multipartite ${P}$-partitions and inner products of skew
  {S}chur functions, Combinatorics and algebra (Boulder, Colo., 1983)
  (Providence, RI), Amer. Math. Soc., 1984, pp.~289--317. 
\bibitem{LR98}
J.-L. Loday and M.O. Ronco, Hopf algebra of the planar
  binary trees, Adv. Math. 139 (1998) no.~2, 293--309.
\bibitem{Mac} I.G. Macdonald, Symmetric functions and Hall polynomials, 2nd ed., Oxford University Press, 1995.
\bibitem{Malv}
C. Malvenuto, Produits et coproduits des fonctions
quasi-sym\'etriques et   de l'alg\`ebre des descents, no.~16, Laboratoire de
combinatoire et   d'informatique math\'ematique {(LACIM)}, Univ.~du Qu\'ebec \`a
Montr\'eal,   Montr\'eal, 1994.
\bibitem{StV1} R.P. Stanley, Enumerative combinatorics, Volume 1, Cambridge Studies in Advanced Mathematics 49 Cambridge University Press, Cambridge, 2012.
\bibitem{Stem97}
J.R. Stembridge, Enriched ${P}$-partitions, Trans. Amer. Math. Soc.
  349 (1997) no.~2, 763--788.
\bibitem{Ze} A.V. Zelevinsky, Representations of finite classical groups, Lecture Notes in Mathematics 869 Springer-Verlag, Berlin-New York, 1981.
\end{thebibliography}
\end{document}